\documentclass[11pt, reqno]{amsart}

\usepackage[T1]{fontenc}
\usepackage{microtype}
\usepackage{lmodern}

\usepackage{amssymb,mathtools}
\usepackage{colonequals}

\usepackage{graphicx}
\usepackage{booktabs}

\usepackage{enumitem}

\usepackage[
    colorlinks=true,
    linkcolor=black,
    citecolor=blue,
    urlcolor=blue
]{hyperref}
\let\oldeqref\eqref
\renewcommand{\eqref}[1]{{\hypersetup{linkcolor=blue}\oldeqref{#1}}}
\usepackage{orcidlink}

\usepackage{aliascnt}
\newtheorem{theorem}{Theorem}[section]
\newaliascnt{lemma}{theorem}
\newtheorem{lemma}[lemma]{Lemma}
\aliascntresetthe{lemma}
\newaliascnt{corollary}{theorem}
\newtheorem{corollary}[corollary]{Corollary}
\aliascntresetthe{corollary}
\newaliascnt{conjecture}{theorem}
\newtheorem{conjecture}[conjecture]{Conjecture}
\aliascntresetthe{conjecture}
\newaliascnt{claim}{theorem}
\newtheorem{claim}[claim]{Claim}
\aliascntresetthe{claim}
\newaliascnt{proposition}{theorem}
\newtheorem{proposition}[proposition]{Proposition}
\aliascntresetthe{proposition}

\theoremstyle{definition}
\newaliascnt{problem}{theorem}

\aliascntresetthe{problem}
\newaliascnt{definition}{theorem}

\aliascntresetthe{definition}

\theoremstyle{remark}
\newaliascnt{example}{theorem}
\newtheorem{example}[example]{Example}
\aliascntresetthe{example}
\newaliascnt{remark}{theorem}
\newtheorem{remark}[remark]{Remark}
\aliascntresetthe{remark}

\usepackage[capitalise]{cleveref}
\crefname{conjecture}{conjecture}{conjectures}
\Crefname{conjecture}{Conjecture}{Conjectures}
\DeclareMathOperator{\Sym}{Sym}
\DeclareMathOperator{\wt}{wt}
\DeclareMathOperator{\disc}{disc}
\newcommand{\ZZ}{\mathbb{Z}}
\newcommand{\RR}{\mathbb{R}}

\title[Supersingular Curves Without Inseparable Small Endomorphisms]{Supersingular Elliptic Curves Without Inseparable Small Degree Endomorphisms}
\date{September 3, 2026}

\author[Nicolas Swanson]{Nicolas Swanson\,\orcidlink{0009-0009-0770-7213}}
\address{Department of Combinatorics and Optimization,
University of Waterloo, Waterloo, Ontario, Canada}
\email{nswanson@uwaterloo.ca}

\begin{document}
\begin{abstract}
For a fixed characteristic $p$, let $\delta(p)$ denote the smallest degree needed to guarantee the existence of an isogeny from any supersingular curve $E$ to its Frobenius conjugate $E^{(p)}$.
We translate existence questions concerning $\delta(p)$ into questions about positive definite integral ternary quadratic forms and use Voronoi's theory of perfect forms in dimension three to characterize when $\delta(p)$ is nearly maximal.
For sufficiently large primes, we show that $\delta(p)$ attains its upper bound precisely when $p$ is represented by one of nine explicit cubic polynomials, reducing the infinitude of such primes to whether one of these cubics takes prime values infinitely often.
We also prove that, for almost all primes, $\delta(p)$ lies at least $p^{1/6-o(1)}$ below its upper bound.
\end{abstract}

\maketitle

Let $p$ be prime, and let $E/\overline{\mathbb F}_p$ be a supersingular elliptic curve over an algebraically closed field of characteristic $p$. Computing the endomorphism ring of $E$ is a central problem in isogeny-based cryptography. A useful algorithmic approach is to search for inseparable endomorphisms \cite{fuselier_computing_2025}. That is, let $E^{(p)}$ denote the Frobenius conjugate of $E$, and consider the Frobenius isogeny
\[
    F\colon E\longrightarrow E^{(p)},
    \qquad (x,y)\longmapsto(x^p,y^p),
\]
which has degree $p$. If one can find a separable isogeny $\psi\colon E^{(p)}\to E$, then postcomposing $F$ with $\psi$ yields an inseparable endomorphism $\psi\circ F$ of degree $p\deg(\psi)$. Conversely, every inseparable endomorphism of $E$ factors through $F$ in this way. Thus, for algorithms using this strategy, the search for a small inseparable endomorphism reduces to finding a small degree isogeny between $E$ and $E^{(p)}$; by duality, either direction suffices.

In this setting, Aubry, Oyono, and Vincent \cite{aubry_minimal_2026} introduced
\[
    \delta_E\colonequals\min_{\varphi\colon E\to E^{(p)}}\deg(\varphi)
    \qquad\text{and}\qquad
    \delta(p)\colonequals\max_{E/\overline{\mathbb F}_p\text{ s.s.}}\delta_E.
\]
The quantity $\delta_E$ is the least auxiliary degree needed to obtain an inseparable endomorphism through Frobenius, while $\delta(p)$ records the worst case among curves in characteristic $p$. Understanding the asymptotic behavior of $\delta(p)$ is therefore important for cryptographic algorithms that exploit the Frobenius isogeny. Aubry, Oyono, and Vincent made progress toward this understanding by establishing the following general upper bound
\cite[Theorem~1.1]{aubry_minimal_2026}:
\[
    \delta(p)\le\left\lfloor\sqrt[3]{p/2}\right\rfloor.
\]
For comparison, Elkies gives a uniform $O(p^{2/3})$ bound for the least degree of a non-scalar endomorphism of an arbitrary supersingular elliptic curve \cite[Section~4]{elkies_existence_1987}.
Wesolowski subsequently demonstrated the significance of this smaller $O(p^{1/3})$ bound by heuristically solving the endomorphism ring problem in expected time and memory $p^{1/3+o(1)}$ \cite[Corollary~1.2 and Section~1.4]{wesolowski_supersingular_2026}.

Beyond their algorithmic applications to endomorphism-ring computation, the quantities $\delta_E$ and $\delta(p)$ raise several natural arithmetic questions. Which values of $\delta_E$ can occur in a fixed characteristic? Is the upper bound for $\delta(p)$ attained for infinitely many primes? More broadly, how is $\delta(p)$ distributed as $p$ varies? In this work, we make progress toward each of these questions:
\begin{itemize}
    \item every finite set of positive integers occurs simultaneously as values of $\delta_E$ in infinitely many prime characteristics,
    \item the upper bound for $\delta(p)$ is attained for infinitely many primes if and only if at least one of nine explicit cubic polynomials takes prime values infinitely often, and
    \item for almost all primes $p$, the difference between $\delta(p)$ and its upper bound is at least $p^{1/6-o(1)}$.
\end{itemize}

All three results come from combining a classical correspondence between maximal quaternion orders and ternary quadratic forms with Aubry and Vincent's characterization of inseparable endomorphism degrees \cite[Theorem~1.1]{aubry_inseparable_2026}.
Namely, a positive integer $n$ occurs as $\delta_E$ for some supersingular elliptic curve $E/\overline{\mathbb F}_p$ if and only if there is a positive definite integral ternary quadratic form with minimum $n$ and discriminant $p$.

In \Cref{sec:existence}, we establish the criterion above. We then show that every sufficiently large prime in a suitable arithmetic progression is the common discriminant of ternary quadratic forms realizing any prescribed finite collection of positive minima. Applying the criterion to this construction proves and strengthens \cite[Conjectures~6.3 and~6.4]{aubry_minimal_2026}.

In \Cref{sec:near-perfect-forms}, we turn to ternary quadratic forms whose minima are close to the largest values permitted by their discriminants. The main technical work of the paper is an explicit classification and parameterization of these forms. Combined with the criterion above, this determines the supersingular curves $E$ for which $\delta_E$ is near-maximal, and hence the primes $p$ for which $\delta(p)$ is near-maximal.

In \Cref{sec:wisde-primes}, we use this classification to obtain finite, computable families of cubic polynomials that parameterize the prime characteristics in which $\delta_E$ can be near-maximal. Their effective enumeration yields an algorithm, polynomial in the bit length of $p$, for computing $\delta(p)$ when it is near-maximal\footnote{We provide an implementation of this algorithm in SageMath \cite{sagemath}, available in \cite{swanson_computational_2026}.}.

Restricting to the maximal case, in which there exists a supersingular elliptic curve $E/\overline{\mathbb F}_p$ satisfying $\delta_E=\delta(p)=\lfloor\sqrt[3]{p/2}\rfloor$,
we obtain a particularly explicit classification. Following \cite[Definition~6.1]{aubry_minimal_2026}, we call a prime with this property \emph{wisde}, short for \emph{without inseparable small degree endomorphisms}. Specifically, we show that every sufficiently large prime $p$ is wisde if and only if $p=P_i(n)$ for some $n\in\ZZ$ and some $i\in\{1,\ldots,9\}$, where the nine cubic polynomials $P_1,\ldots,P_9$ are displayed in \eqref{eq:cubic-families}. 
Consequently, there are infinitely many wisde primes if and only if one of these cubics takes prime values infinitely often. Establishing this for even one of the nine cubics lies beyond current number-theoretic techniques, but the Bateman--Horn conjecture \cite{bateman_heuristic_1962} predicts an affirmative answer and gives an explicit asymptotic for the number of wisde primes up to $X$.

Finally, in \Cref{sec:wisde-deficits}, we study the typical gap between $\delta(p)$ and its upper bound. The computations in \cite[Section~5]{aubry_minimal_2026} led to the conjecture that this gap is $O(\log p)$ \cite[Conjecture~6.8]{aubry_minimal_2026}. We show instead that, for a set of primes $p$ of relative density one,
\[
    \left\lfloor\sqrt[3]{p/2}\right\rfloor-\delta(p)
    >p^{1/6}(\log p)^{-1/2-\varepsilon}
\]
for every $\varepsilon>0$. In particular, the gap is at least $p^{1/6-o(1)}$ for almost all primes.
Although this shows that the typical gap grows at least as a positive power of $p$, it does not change the $p^{1/3}$ asymptotic relevant to the endomorphism-ring algorithms mentioned above. In fact, our heuristic analysis leads us to conjecture that the gap is $p^{1/6+o(1)}$ for almost all primes. Thus, although primes for which $\delta(p)$ attains its upper bound are sparse, we expect that upper bound to remain asymptotically sharp for almost all primes.

\section*{Acknowledgments}
The author would like to thank David Jao and Bruno Sterner for their valuable editorial and expository feedback. 
The author also thanks Travis Morrison for helpful discussions that shaped the framing and presentation of Aubry, Oyono, and Vincent's work in \Cref{sec:existence}.
The author acknowledges support of the NSERC Alliance Consortia Quantum Grant ``Accelerating the transition to quantum-resistant cryptography'' (ALLRP 578463--2022), the NSERC Discovery Grant program, the NSERC Canada Graduate Research Scholarship program, and CryptoWorks21.

\section{Preliminaries}
\label{sec:preliminaries}
\subsection{Rank 3 Lattices and Ternary Quadratic Forms}
Two ternary quadratic forms $Q$ and $Q'$ are \emph{integrally equivalent} if $Q(\vec x)=Q'(U\vec x)$ for some \emph{unimodular matrix} $U\in\operatorname{GL}_3(\ZZ)$. There are several notions of reduction that select a unique representative from each integral equivalence class. Let $Q$ be an arbitrary positive definite integral ternary quadratic form and write
\[
    Q=aX^2+bY^2+cZ^2+rYZ+sXZ+tXY.
\]
We call $a,b,c\in\ZZ$ the \emph{diagonal coefficients} of $Q$ and $r,s,t\in\ZZ$ its \emph{cross coefficients}. We say that $Q$ is \emph{Eisenstein-reduced} if
\[
    \begin{cases}
        0<a\le b\le c,\\
        |s|, |t| \le a \text{ and } |r| \le b,\\
        a+b+r+s+t\ge0, \\
        r,s,t \text{ all have the same sign},
    \end{cases}
\]
together with the boundary conditions in \cite[Part~III, pp.~164--171]{dickson_studies_1930}, which ensure uniqueness in the equality cases. 
By changing the signs of $Y$ and $Z$ when necessary, we may replace the common-sign condition on $r,s,t$ by $s,t\ge0$, while allowing $r$ to be negative. This makes the last two conditions above become
\[
\begin{cases}
    s+t-r\le a+b,\\
    s,t\ge0, \text{ while } r \text{ may be negative}.
\end{cases}
\]
We call the resulting representative \emph{sign-normalized Eisenstein-reduced}, and note that it is obtained uniquely from an Eisenstein-reduced representative by a simple sign change.


A rank $3$ quadratic lattice $(L,Q)$ is a rank $3$ lattice $L$ equipped with an integral ternary quadratic form $Q$. After fixing a basis, there is a unique symmetric matrix $G_Q$ and a unique integral matrix $A_Q\colonequals 2G_Q$ such that
\[
Q(\vec x)=\vec x^{\,T}G_Q\vec x=\frac{1}{2}\vec x^{\,T}A_Q\vec x.
\]
If $Q$ is positive definite, it defines a norm form and an inner product on $L$ for which $G_Q$ is the Gram matrix.
We take the \emph{discriminant} of $Q$ to be
\[
    \disc(Q)\colonequals4\det(G_{Q}) = \frac{1}{2}\det(A_{Q}).
\]
We use lowercase $\disc(Q)$ to distinguish our convention, as used in \cite[Chapter~22]{voight_quaternion_2021}, from the $\operatorname{Discr}(Q) = \det(G_Q)$ used by \cite[Section~4.1]{aubry_minimal_2026} and others. 

The \emph{successive minima} of a quadratic lattice $(L, Q)$ with $Q$ positive definite are defined by
\[
    D_i=\min\left\{D : \operatorname{rank}\{x\in L:Q(x)\le D\}\ge i\right\}.
\]
A basis $\{\beta_1, \beta_2, \beta_3\}$ of $L$ with $Q(\beta_i)=D_i$ is called a \emph{basis of successive minima}, and in rank $3$, such a basis always exists.
If $Q$ is Eisenstein-reduced or sign-normalized Eisenstein-reduced, then the standard basis of $\ZZ^3$ is a basis of successive minima \cite[Part~III, pp.~164--171]{dickson_studies_1930}. Thus, its diagonal coefficients $a,b,c$, equivalently the diagonal entries of $G_Q$, are the successive minima $D_1,D_2,D_3$. We define the \emph{minimum} of a positive definite integral ternary quadratic form $Q$ by
\[
    \min(Q)\colonequals\min_{\vec x\in\ZZ^3\setminus\{0\}}Q(\vec x).
\]
The \emph{minimal vectors of $Q$ modulo negation} are denoted by
\[
    S(Q)\colonequals\{\vec x\in\ZZ^3:Q(\vec x)=\min(Q)\}/\{\pm1\}.
\]
Thus, for an Eisenstein-reduced or sign-normalized Eisenstein-reduced form $Q$, we have $\min(Q)=D_1=a$, and the representatives of $S(Q)$ are precisely the vectors attaining the first successive minimum. Also, for such a form, the \emph{Hadamard bound} and the \emph{Hermite bound} state, respectively, that
\begin{equation}
    \label{eq:hadamard-hermite}
    \det(G_Q)\le D_1D_2D_3\le 2\det(G_Q).
\end{equation}

\subsection{Voronoi Theory and the Perfect Form}

Equality in the right-hand Hermite inequality in \eqref{eq:hadamard-hermite} holds if and only if the lattice is integrally equivalent to a scaled copy of the $A_3$ root lattice \cite[Theorem~2.6.8, p.~50, and Theorem~6.2.1, pp.~194--195]{martinet_perfect_2003}. The $A_3$ root lattice is the rank $3$ sublattice of $\RR^4$ generated by $\{e_i-e_j:1\le i\ne j\le4\}$. The standard basis and inner product on $\RR^4$ induce (after scaling by $1 / 2$) the positive definite integral ternary quadratic form
\[
    Q_0\colonequals X^2+Y^2+Z^2+YZ+XZ+XY.
\]

This extremal lattice and form motivate the Voronoi-theoretic framework we describe next. For any integral lattice $L$, write $L^\vee = \operatorname{Hom}_\ZZ(L, \ZZ)$ for its dual. The space of integral ternary quadratic forms is $\Sym^2((\ZZ^3)^\vee)$. Up to positive scaling and integral equivalence, the form $Q_0$ is the unique \emph{perfect} ternary quadratic form, meaning
\[
    \operatorname{span}_{\RR}\{\vec x\vec x^{\, T}:\vec x\in S(Q_0)\}
    = \Sym^2(\RR^3) \cong \left(\Sym^2((\ZZ^3)^\vee) \otimes_{\ZZ}\RR\right)^\vee.
\]
This characterization is all we need from the general theory of perfect forms. Each element of $S(Q_0)$ is a pair $\{\pm\vec v\}$, and either representative gives the same tensor $\vec v\vec v^{\,T}$. Whenever explicit representatives are needed, we use
\begin{equation}
    \label{eq:fixedvi}
    (\vec v_1,\ldots,\vec v_6)
    =
    \left(
        \begin{bmatrix} 1\\0\\0 \end{bmatrix},
        \begin{bmatrix} 0\\1\\0 \end{bmatrix},
        \begin{bmatrix} 0\\0\\1 \end{bmatrix},
        \begin{bmatrix} 1\\-1\\0 \end{bmatrix},
        \begin{bmatrix} 1\\0\\-1 \end{bmatrix},
        \begin{bmatrix} 0\\1\\-1 \end{bmatrix}
    \right).
\end{equation}
The tensors $\vec v_i\vec v_i^{\,T}$ form a basis of $\Sym^2(\ZZ^3)$; we denote its dual basis in $\Sym^2((\ZZ^3)^\vee)$ by $\eta_1,\ldots,\eta_6$, so that $\eta_j(\vec v_i)=\delta_{ji}$. Thus, every integral ternary quadratic form $h$ can be expressed uniquely in terms of its evaluations at the minimal vectors of $Q_0$ as
\[
    h=\sum_{i=1}^6 h(\vec v_i)\eta_i.
\]
We denote the coordinates of $h$ in this basis by $\lambda=(h(\vec v_1),\ldots,h(\vec v_6))$. Conversely, given $\lambda=(\lambda_1,\ldots,\lambda_6)\in\ZZ^6$, we define $h_\lambda$ as the unique integral ternary quadratic form with coordinates $\lambda$ in the $\eta_i$ basis. Explicitly, with the representatives chosen in \eqref{eq:fixedvi},
\begin{align}
    \label{eq:hlam-expansion}
    h_\lambda &= \lambda_1X^2+\lambda_2Y^2+\lambda_3Z^2+(\lambda_2+\lambda_3-\lambda_6)YZ \notag\\
    &\quad+(\lambda_1+\lambda_3-\lambda_5)XZ+(\lambda_1+\lambda_2-\lambda_4)XY.
\end{align}
For a coordinate vector $\lambda\in\ZZ^6$, define its \emph{weight} and \emph{minimum} by $\wt(\lambda)\colonequals\sum_i\lambda_i$ and $\lambda_{\min}\colonequals\min_i\lambda_i$, respectively.

In the realization of $A_3$ by differences of the four standard basis vectors $\vec e_1,\ldots,\vec e_4$, permutations of these vectors, together with negation, give
$\operatorname{Aut}(Q_0)\cong\operatorname{Aut}(A_3)\cong S_4\times\{\pm1\}$ \cite[pp.~147, 150]{scott_construction_1966}. 
Because negation fixes every element of $S(Q_0)$, the action of $\operatorname{Aut}(Q_0)$ on $S(Q_0)$ factors through $\operatorname{Aut}(Q_0)/\{\pm1\}\cong S_4$.
Identifying $S(Q_0)$ with the unordered pairs $\{a,b\}\subseteq\{1,2,3,4\}$ via $\{\pm(\vec e_a-\vec e_b)\}\leftrightarrow\{a,b\}$, we write $\lambda_{\{a,b\}}$ for the corresponding coordinate.
This implies that the parameterization $\lambda\mapsto h_\lambda$ is $S_4$-equivariant: if $g_\pi\in\operatorname{Aut}(Q_0)$ represents $\pi\in S_4$ and
\[
    (\pi\cdot\lambda)_{\{\pi(a),\pi(b)\}}=\lambda_{\{a,b\}},
\]
then
\begin{equation}
    \label{eq:evaluation-equivariance}
    h_{\pi\cdot\lambda}=h_\lambda\circ g_\pi^{-1}.
\end{equation}

\subsection{The Gross Lattice}
Let $B_{p,\infty}$ denote the quaternion algebra over $\mathbb{Q}$ ramified at $p$ and at $\infty$. By the Deuring correspondence, the endomorphism ring of any supersingular elliptic curve over $\overline{\mathbb F}_p$ is isomorphic to a maximal order $\mathcal O\subset B_{p,\infty}$ \cite[Theorem~42.1.9]{voight_quaternion_2021}. Under this correspondence, an endomorphism $\alpha\in\operatorname{End}(E)$ maps to a quaternion $q\in\mathcal O$ satisfying $\operatorname{trd}(\alpha)=\operatorname{trd}(q)$ and $\deg(\alpha)=\operatorname{nrd}(q)$.

A central object in this setting is the Gross lattice associated with a supersingular elliptic curve. Define $\tau\colon B_{p,\infty}\to B_{p,\infty}$ by $\tau(x)=2x-\operatorname{trd}(x)$. For an order $\mathcal O$ in $B_{p,\infty}$, its \emph{Gross lattice} is $\mathcal O^T\colonequals\tau(\mathcal O)$. In particular, if $E$ is supersingular and $\operatorname{End}(E)\cong\mathcal O$, then $\mathcal O^T$ is the Gross lattice associated with $E$. The map $\tau$ induces an isomorphism of lattices $\mathcal O/\ZZ\cong\mathcal O^T$. Thus, $\mathcal O^T$ represents the nonscalar endomorphisms of $E$ under the Deuring correspondence. It has rank $3$ and lies in the trace-zero sublattice of $\mathcal O$.

\section{Existence of Supersingular Curves at Prescribed Distances from Their Frobenius Conjugates}
\label{sec:existence}

\subsection{An Integral Ternary-Form Reformulation}

As discussed in the introduction, finding an isogeny $\varphi\colon E\to E^{(p)}$ of degree $\ell$ is equivalent, via composition with or factorization through the Frobenius isogeny, to finding an inseparable endomorphism of $E$ of degree $\ell p$. Aubry and Vincent prove the following result \cite[Theorem~1.1]{aubry_inseparable_2026}.
\begin{theorem}
    \label{thm:av_main}
    Let $p$ be prime and let $E/\overline{\mathbb F}_p$ be a supersingular elliptic curve with associated Gross lattice $\mathcal O^T$. Then for any positive integer $\ell$, there exists a trace-zero endomorphism of $E$ of degree $\ell p$ if and only if $\mathcal O^T$ contains a rank $2$ sublattice of determinant $4\ell p$.
\end{theorem}

Although an endomorphism of degree $\ell p$ need not have trace zero in general, Aubry, Oyono, and Vincent observe that it is automatic when $\ell p<p^2/4$ \cite[discussion immediately preceding Proposition~3.1]{aubry_minimal_2026}. For any two supersingular elliptic curves, the least degree of an isogeny between them is $O(p^{2/3})$ \cite[Section~4]{elkies_existence_1987}. Consequently, for all sufficiently large $p$, \Cref{thm:av_main} applies to the inseparable endomorphism corresponding to a least-degree isogeny $E\to E^{(p)}$; handling the remaining primes separately yields the following characterization of $\delta_E$ \cite[Proposition~3.1]{aubry_minimal_2026}.

\begin{proposition}
    \label{prop:aov_delta_E}
    Let $p$ be prime and let $E/\overline{\mathbb F}_p$ be a supersingular elliptic curve with associated Gross lattice $\mathcal O^T$. Then
    \[
        \delta_E = \min_{\Lambda \subset \mathcal O^T} \frac{\det(\Lambda)}{4p},
    \]
    where the minimum runs over all rank $2$ sublattices $\Lambda$ of the Gross lattice $\mathcal O^T$.
\end{proposition}

To connect this characterization with ternary quadratic forms, we interpret the normalized determinants in \Cref{prop:aov_delta_E} as values of a single quadratic form on $\bigwedge^2\mathcal O^T$.
Recall that a rank $2$ sublattice $\Lambda\subset\mathcal O^T$ with basis $\alpha_1,\alpha_2$ determines the vector $\alpha_1\wedge\alpha_2\in\bigwedge^2\mathcal O^T$, well-defined up to sign. Since $\mathcal O^T$ has rank $3$, so does $\bigwedge^2\mathcal O^T$. The wedge product also satisfies a universal property: every alternating bilinear map on $\mathcal O^T$ factors uniquely through a linear map from $\bigwedge^2\mathcal O^T$.

Let $[\mathcal O,\mathcal O]$ denote the two-sided ideal generated by commutators in the maximal order $\mathcal O$, and let $[\mathcal O,\mathcal O]^0$ denote its trace-zero sublattice. The normalized commutator pairing on $\mathcal O^T$ studied by Aubry and Vincent \cite{aubry_inseparable_2026} is alternating and bilinear with values in $[\mathcal O,\mathcal O]^0$:
\[
    (\alpha_1,\alpha_2)\longmapsto
    \frac{\alpha_1\alpha_2-\alpha_2\alpha_1}{4}.
\]
The universal property therefore induces a linear map
\[
    \bigwedge^2\mathcal O^T\longrightarrow[\mathcal O,\mathcal O]^0,
    \qquad \alpha_1\wedge\alpha_2\longmapsto\frac{[\alpha_1,\alpha_2]}{4}.
\]
By \cite[Corollary~3.2]{aubry_inseparable_2026}, this map is an isomorphism. Moreover, \cite[Proposition~2.2]{aubry_inseparable_2026} shows that the determinant and reduced norm are equivalent through this isomorphism up to a factor of $4$:
\[
\det(\alpha_1, \alpha_2) = 4\operatorname{nrd}\left (\frac{[\alpha_1, \alpha_2]}{4} \right ).
\]

The significance of this isomorphism is that $[\mathcal O,\mathcal O]$ is local at $p$: it is the unique two-sided prime ideal of any maximal order above the ramified prime $p$. Consequently, the reduced norm of every element of $[\mathcal O,\mathcal O]$ is divisible by $p$, and $\operatorname{nrd}/p$ defines an integral ternary quadratic form on $[\mathcal O,\mathcal O]^0$. Pulling this form back to $\bigwedge^2\mathcal O^T$ through the isomorphism above yields an integral ternary quadratic form $Q$ such that
\[
    Q(\alpha_1\wedge\alpha_2)
    =\frac{\det(\alpha_1,\alpha_2)}{4p}.
\]
Every nonzero vector in $\bigwedge^2\mathcal O^T$ arises from a rank $2$ sublattice of $\mathcal O^T$, so $Q$ takes precisely the normalized determinant values appearing in \Cref{prop:aov_delta_E}. Hence, $\delta_E=\min(Q)$.

To describe $Q$ explicitly, let $G$ be the Gram matrix of $\operatorname{nrd}$ on $\mathcal O^T$. With this notation, the induced determinant form on $\bigwedge^2\mathcal O^T$ has Gram matrix $\operatorname{adj}(G)$. Since $Q$ is this form scaled by $1/(4p)$, we obtain
\begin{equation}
    \label{eq:wedge-form}
    Q(\vec x) = \frac{1}{4p}\vec x^{\, T}\operatorname{adj}(G)\vec x.
\end{equation}
The integrality of $Q$ is striking: for an arbitrary Gram matrix $G$ satisfying $\det(G)=4p^2$, the right-hand side of \eqref{eq:wedge-form} need not define an integral form. However, in the case of the Gross lattice, integrality is forced by the identification of $Q$ with $\operatorname{nrd}/p$ on $[\mathcal O,\mathcal O]^0$, and its discriminant follows directly from equation~\eqref{eq:wedge-form}. Namely, since $\det(G)=4p^2$ by \cite[Proposition~4.1]{aubry_minimal_2026},
\[
    \disc(Q)
    =4\det\left(\frac{\operatorname{adj}(G)}{4p}\right)
    =\frac{4\det(G)^2}{(4p)^3}
    =p.
\]

Thus, interpreted in the language of ternary quadratic forms, the results of \cite{aubry_inseparable_2026,aubry_minimal_2026} give the forward direction of the following proposition.

\begin{proposition}
    \label{prop:curve_tqf}
    Let $p$ be prime and let $n$ be a positive integer. There exists a supersingular elliptic curve $E$ over $\overline{\mathbb F}_p$ with $\delta_E=n$ if and only if there exists a positive definite integral ternary quadratic form $Q$ such that
    \[
        \min(Q) = n \quad \text{ and } \quad \disc(Q) = p.
    \]
\end{proposition}
The converse produces a curve $E$ from a form $Q$ through two classical correspondences. The even Clifford algebra construction associates to any positive definite integral ternary quadratic form $Q$ of discriminant $p$ a maximal order $\mathcal O\subset B_{p,\infty}$ \cite[Main Theorem~22.1.1]{voight_quaternion_2021}, and the Deuring correspondence realizes $\mathcal O$ as the endomorphism ring of a supersingular elliptic curve $E/\overline{\mathbb F}_p$. To conclude that $\delta_E=\min(Q)$, one must additionally verify that the normalized determinant form associated with the Gross lattice $\mathcal O^T$ is integrally equivalent to the original form $Q$. We give the key compatibility calculation below and defer a formal proof of both directions to Appendix~\ref{app:curve_tqf}.

For the maximal order $\mathcal O$ associated with $Q$ above, Voight's reduced-norm formula \cite[(22.3.18)]{voight_quaternion_2021} gives, in the bases arising from the Clifford algebra construction,
\[
    G_{\mathcal O^T}=4\operatorname{adj}(G_Q).
\]
The adjugate here arises from the Clifford algebra norm formula, whereas in \eqref{eq:wedge-form} it arose from the determinant form on $\bigwedge^2\mathcal O^T$. Nonetheless, taking adjugates of both sides, scaling by $1/(4p)$, and using $\det(G_Q)=\disc(Q)/4=p/4$ yields
\[
    \frac{1}{4p}\operatorname{adj}(G_{\mathcal O^T})
    =\frac{1}{4p}\operatorname{adj}\left(4\operatorname{adj}(G_Q)\right)
    =\frac{4\det(G_Q)}{p}G_Q
    =G_Q.
\]
Consequently, the normalized determinant form associated with $\mathcal O^T$ has Gram matrix $G_Q$ in the corresponding basis and is therefore integrally equivalent to $Q$.

\begin{remark}
    \label{rem:curve_tqf_effective}
    The reverse implication in \Cref{prop:curve_tqf} is constructive. Namely, starting from $Q$, the even Clifford algebra construction directly produces an explicit basis for the associated maximal order $\mathcal O\subset B_{p,\infty}$. Thus, realizing $Q$ by an explicit curve reduces to the constructive Deuring problem: computing a supersingular elliptic curve $E/\mathbb F_{p^2}$ whose endomorphism ring is isomorphic to $\mathcal O$. Once such a curve is computed, \Cref{prop:curve_tqf} ensures that $\delta_E=\min(Q)$. Under the Generalized Riemann hypothesis (GRH), the constructive Deuring problem can be solved in probabilistic polynomial time in $\log p$ \cite{cremona_deuring_2024}.
\end{remark}

\begin{remark}
    \label{remark:aov_counterexample}
    The minimum criterion in \Cref{prop:curve_tqf} gives a direct explanation for an exceptional example arising in \cite[Section~5.1]{aubry_minimal_2026}. In their search for wisde primes, Aubry, Oyono, and Vincent derive necessary determinant and coefficient constraints to enumerate candidate Gram matrices for the normalized degree lattice of a curve satisfying
    \[
        \delta_E=\left\lfloor\sqrt[3]{p/2}\right\rfloor.
    \]
    For $p=22273$, their sieve produces
    \[
        G=
        \begin{pmatrix}
            22 & 15/2 & 8\\
            15/2 & 22 & -9\\
            8 & -9 & 23
        \end{pmatrix}.
    \]
    Although $G$ satisfies $\det(G)=p/4$, they report that this candidate does not yield a supersingular elliptic curve with $\delta_E=22$. \Cref{prop:curve_tqf} exactly explains this: the associated integral ternary quadratic form is
    \[
        Q(X,Y,Z)=22X^2+22Y^2+23Z^2-18YZ+16XZ+15XY,
    \]
    so $\disc(Q)=p$. However, $Q$ is not Eisenstein-reduced, and a direct Eisenstein reduction shows that $\min(Q)=Q(1,-1,-1)=18$, not $22$. By \Cref{prop:curve_tqf}, $Q$ therefore corresponds to a supersingular elliptic curve with $\delta_E=18$.
\end{remark}

\subsection{Constructing Curves at Prescribed Distances}

We have now restated the problem of finding a supersingular curve $E/\overline{\mathbb F}_p$ with $\delta_E=n$ as the problem of finding a positive definite integral ternary quadratic form $Q$ with $\min(Q)=n$ and $\disc(Q)=p$. Our first result uses this criterion to show that every finite set of values can occur simultaneously as $\delta_E$ for a fixed characteristic. 

\begin{theorem}
    \label{thm:every-delta}
    For every finite nonempty set $S\subset\ZZ_{>0}$, there are infinitely many primes $p$ such that for every $m\in S$, there is a supersingular elliptic curve $E/\overline{\mathbb F}_p$ with $\delta_E=m$.
\end{theorem}

The idea is to construct, for each $m\in S$, a one-parameter family of forms $Q_{m,t}$ such that
\[
    \disc(Q_{m,t})=A_mt-b_m
    \quad\text{ and }\quad
    \min(Q_{m,t})=m
\]
whenever $t$ is sufficiently large. Let $M\colonequals\prod_{m\in S}A_m$. By making the $A_m$ pairwise coprime, the Chinese remainder theorem (CRT) produces a residue class $a\in\ZZ/M\ZZ$ satisfying $a\equiv-b_m\pmod{A_m}$ for every $m\in S$. Dirichlet's theorem on arithmetic progressions guarantees infinitely many primes in this residue class. For every sufficiently large such prime $p$, setting $t_m\colonequals(p+b_m)/A_m$ gives $\disc(Q_{m,t_m})=p$ and $\min(Q_{m,t_m})=m$ for every $m\in S$.

\begin{proof}
    For each $m\in S$, consider the family
    \[
        Q_{m,t}(X,Y,Z)=mX^2+b_mY^2+tZ^2+XZ+XY,
    \]
    where $b_m$ will be specified shortly. A direct calculation gives
    \[
        \disc(Q_{m,t})=(4mb_m-1)t-b_m.
    \]
    Thus, writing $A_m\colonequals4mb_m-1$, we have $\disc(Q_{m,t})=A_mt-b_m$.

    To apply the CRT, we choose the $b_m$ so that the integers $A_m=4mb_m-1$ are pairwise coprime. This can be done by Dirichlet's theorem: for each $m\in S$, select a distinct prime
    \[
        A_m\equiv -1\pmod{4m}
    \]
    and set $b_m=(A_m+1)/(4m)$. Then $b_m$ is an integer, and the $A_m$ are pairwise coprime. Since we may take $A_m$ arbitrarily large, we may also arrange that $A_m>4m^2$, so that $b_m>m$.

    Let $M\colonequals\prod_{m\in S}A_m$. The CRT gives a residue class $a\in\ZZ/M\ZZ$ such that
    $a\equiv -b_m\pmod{A_m}$ for every $m\in S$.
    Because $0<b_m<A_m$, we have $\gcd(a,M)=1$, so Dirichlet's theorem yields infinitely many primes $p\equiv a\pmod M$. Choose any sufficiently large such prime $p$ with $p>mA_m$ for every $m\in S$. For each $m\in S$, define
    \[
        t_m\colonequals\frac{p+b_m}{A_m},
        \qquad\text{so that}\qquad
        \disc(Q_{m,t_m})=A_mt_m-b_m=p.
    \]

    It remains to show that $\min(Q_{m,t_m})=m$, which will follow from $b_m>m$ and $p>mA_m$. Completing the square first in $X$ and then in $Y$ gives
    \[
        Q_{m,t_m}(X,Y,Z)
        =m\left(X+\frac{Y+Z}{2m}\right)^2
        +\frac{A_m}{4m}\left(Y-\frac{Z}{A_m}\right)^2
        +\frac{p}{A_m}Z^2.
    \]
    If $Z\ne0$, then the choice of $p$ gives 
    \[
    Q_{m,t_m}(X,Y,Z)\ge (p/A_m)Z^2>m.
    \] 
    If $Z=0$ and $Y\ne0$, then the choice of $b_m$ gives
    \[
    Q_{m,t_m}(X,Y,0)\ge \frac{A_m}{4m}Y^2=(b_m-1/(4m))Y^2>m.
    \]
    Finally, if $Y=Z=0$ and $X\ne0$, then 
    \[
    Q_{m,t_m}(X,0,0)=mX^2\ge m.
    \] 
    Since $Q_{m,t_m}(1,0,0)=m$, it follows that $Q_{m,t_m}$ is positive definite with $\min(Q_{m,t_m})=m$. Applying \Cref{prop:curve_tqf} for every $m\in S$ gives the desired curves.
\end{proof}

Taking $S=\{n\}$ in \Cref{thm:every-delta} shows that every positive integer $n$ occurs as $\delta_E$ in infinitely many prime characteristics. Taking $S=\{1,\ldots,k\}$ shows that, for every sufficiently large prime $p$ in a suitable arithmetic progression, there are supersingular curves in characteristic $p$ realizing every $\delta_E$-value from $1$ through $k$. These conclusions prove and strengthen \cite[Conjecture~6.3]{aubry_minimal_2026} and \cite[Conjecture~6.4]{aubry_minimal_2026}, respectively.

The size of the characteristic produced by the proof depends on both the lower bound used to determine the minimum and the CRT modulus. Consider first $S=\{1,\ldots,k\}$, and suppose optimistically that each auxiliary prime $A_m$ can be chosen with $A_m=\Theta(m^2)$. The CRT modulus then satisfies
\[
    M=\prod_{m=1}^kA_m=k^{\Theta(k)}.
\]
The usual heuristic for primes in arithmetic progressions predicts that sampling the resulting progression produces a prime of size $p=k^{\Theta(k)}$. Thus, the rapid growth in the characteristic arises in this proof from requiring one characteristic to realize many prescribed values simultaneously.

\begin{remark}
    \label{rem:param_fam_alg}
    By contrast, the cubic dependence in the singleton case $S=\{n\}$ can be made explicit without requiring $A_n$ to be prime. Choosing $b_n=n+1$ and writing
    \[
        D\colonequals A_n=4n^2+4n-1
    \]
    gives the one-parameter family of ternary quadratic forms
    \[
        Q_{n,t}(X,Y,Z)=nX^2+(n+1)Y^2+tZ^2+XZ+XY.
    \]
    In the proof of \Cref{thm:every-delta}, the $A_m$ were taken prime only to make the moduli pairwise coprime; this is unnecessary when $S=\{n\}$. Since
    \[
        \disc(Q_{n,t})=Dt-(n+1)
        \quad\text{ and }\quad
        \gcd(D,n+1)=1,
    \]
    Dirichlet's theorem gives infinitely many prime discriminants $p=Dt-(n+1)$ as $t$ varies. Since $b_n=n+1>n$, the minimum argument in the proof of \Cref{thm:every-delta} gives $\min(Q_{n,t})=n$ whenever $p>nD$.

    Taking $A_n=D$ avoids the optimistic assumption $A_n=\Theta(n^2)$ in the singleton case, since $D=4n^2+4n-1$ has the required size. This allows us to estimate the size of the prime candidates produced by the construction. To keep $p$ within a constant factor of the required threshold $p>nD$, we may take $t\in[n+2,2(n+1)]$ and search among the candidates
    \[
        p=Dt-(n+1)=\Theta(n^3).
    \]
    The size of $p$ is therefore cubic in $n$, which is optimal up to a constant factor since the Hermite bound forces $p\ge 2n^3$ whenever $\delta_E=n$.

    Dirichlet's theorem guarantees infinitely many primes for each fixed $n$, but it does not show that one occurs in the interval above as $n$ grows. Suppose heuristically that primes of size $p$ are equidistributed among the $\varphi(D)$ reduced residue classes modulo $D$. Since integers of size $p$ have prime density about $1/\log p$, each candidate $p=Dt-(n+1)$ has probability approximately $D/(\varphi(D)\log p)$ of being prime.
    This suggests that one should expect to test $O(\log p)=O(\log n)$ values of $t$ before finding a prime. Moreover, since the interval contains $\Theta(n)=\Theta(p^{1/3})$ candidates, treating their primality as independent predicts that the probability of finding a prime tends to $1$ as $n$ grows. We emphasize that this is only heuristic: we are not aware of a proof of the required estimate for this growing family of moduli, even under GRH.

    For the $256$-bit example in Appendix~\ref{app:param_ex}, our implementation \cite{swanson_computational_2026} finds such a prime quickly. Applying the constructive procedure described in \Cref{rem:curve_tqf_effective} then produces a supersingular elliptic curve $E/\mathbb F_{p^2}$ with the prescribed value $\delta_E=n=\Theta(p^{1/3})$; in this example, $\delta_E\approx 2^{84}$.
\end{remark}

The construction above prescribes $\delta_E$ for one curve in characteristic $p$, but it does not determine the maximum $\delta(p)$ over all curves in that characteristic. That is, it proves that $\delta(p)\ge n$, but another curve may have a larger value of $\delta_E$. Consequently, it does not resolve the conjectured surjectivity of $p\mapsto\delta(p)$ from \cite[Conjecture~6.5]{aubry_minimal_2026}.

\section{Explicit Convergence to the Perfect Form}
\label{sec:near-perfect-forms}

We now turn our attention to positive definite integral ternary quadratic forms for which the Hermite bound
\[
    \disc(Q) \ge 2\min(Q)^3
\]
is nearly an equality. By \Cref{prop:curve_tqf}, this characterization helps identify and compute the supersingular elliptic curves $E$ for which $\delta_E$ is near-maximal. Independently of that application, this section gives an explicit convergence result for near-perfect integral lattices in $\RR^3$.

When equality holds in the Hermite bound, so that $\min(Q)=n$ and $\disc(Q)=2n^3$, the form $Q$ is integrally equivalent to $nQ_0$ \cite[Theorem~III.C, pp.~33--34]{cassels_introduction_1959}.
It is also well known that $Q_0$ is, up to scaling and integral equivalence, the unique perfect ternary form \cite[Theorem~6.2.1, pp.~194--195]{martinet_perfect_2003}.
Accordingly, as $\disc(Q)/(2n^3)$ tends to $1$, one can choose equivalent representatives for which $Q/n$ tends to $Q_0$. Our purpose is to make this convergence explicit when $\disc(Q)<2(n+e)^3$.

Voronoi's theory provides natural coordinates for describing the perturbations of $nQ_0$ that we wish to study. Intuitively, for a representative $Q=nQ_0+h_\lambda$, the coordinates $\lambda_i=h_\lambda(\vec v_i)=Q(\vec v_i)-n$ measure the difference between $Q$ and $nQ_0$ along the six minimal-vector directions of $Q_0$.
The following lemma makes this coordinate description of the convergence to $Q_0$ concrete.

\begin{lemma}
    \label{lem:a3}
    Let $n, e \in \ZZ$ with $n\ge 25e^{2}$ and $e \ge 1$. If $Q$ is a positive definite integral ternary quadratic form with
    \[
        \min(Q) = n \quad \text{ and } \quad \disc(Q) < 2(n + e)^3,
    \]
    then $Q$ is integrally equivalent to $nQ_0+h_\lambda$ for some $\lambda\in\ZZ_{\ge0}^6$ satisfying $\lambda_{\min}=0$. Every coefficient of $h_\lambda$ is bounded in absolute value by $57e$. In particular, $h_\lambda$ belongs to a bounded, finite, enumerable set independent of $n$.
\end{lemma}

\begin{proof}
    Consider the sign-normalized Eisenstein-reduced representative of $Q$ and write
    \[
        Q=nX^2+(n+a_0)Y^2+(n+b_0)Z^2+xYZ+yXZ+zXY,
    \]
    so that $0\le a_0\le b_0$, $0 \le y,z \le n$, and $|x|\le n+a_0$.

    The Hermite inequality in \eqref{eq:hadamard-hermite} and the upper bound on $\disc(Q)$ give
    \[
        n(n+a_0)(n+b_0) \le \frac {\disc(Q)}{2} < (n + e)^3.
    \]
    Hence,
    \[
        0 > (a_0+b_0-3e)n^2 + (a_0b_0-3e^2)n-e^3 \ge (a_0+b_0-3e)n^2 - 3e^2 n-e^3.
    \]
    If $a_0+b_0>3e$, then integrality gives $a_0+b_0-3e\ge1$. The rightmost expression is therefore at least $n^2-3e^2n-e^3$, which is positive for $n\ge25e^2$. This contradiction proves that $a_0+b_0\le3e$ and $a_0\le3e/2$.

    It remains to show that the cross coefficients $(x,y,z)$ differ by only $O(e)$ from those of an integrally equivalent copy of $nQ_0$. The sign-normalized conditions reduce the comparison to the following six equivalent copies of $Q_0$:
    \[
        \mathcal A = \left\{
            (1,1,1),(0,1,1),(1,0,1),(1,1,0),(-1,0,1),(-1,1,0)
        \right\}.
    \]

    Consider the normalized cross coefficients $\bar x=x/n$, $\bar y=y/n$, and $\bar z=z/n$. A direct computation of $\disc(Q)$ verifies that
    \begin{equation}
        \label{eq:disc-with-R}
        \frac{\disc(Q)}{n^3}-2 = R(\bar x,\bar y,\bar z)
        + \frac{4a_0 + 4b_0 - b_0\bar z^2 - a_0\bar y^2}{n}
        + \frac{4a_0b_0}{n^2},
    \end{equation}
    with $R(u,v,w)\colonequals 2-u^2-v^2-w^2+uvw$. The definition here is natural because
    \[
        2 + R(u, v, w) = \disc( X^2 + Y^2 + Z^2 + uYZ + vXZ+ wXY).
    \]
    Because $a_0,b_0\ge0$ and $|\bar y|,|\bar z|\le1$, the term $R(\bar x,\bar y,\bar z)$ is the only potentially negative term in \eqref{eq:disc-with-R}. The upper bound on $\disc(Q)$ gives
    \begin{align}
        \label{eq:a3-discriminant-bound}
        R(\bar x,\bar y,\bar z) \le \frac {\disc(Q)}{n^3}-2 < 2 \frac{(n + e)^3}{n^3} - 2 = \frac{6e}{n}+\frac{6e^2}{n^2}+\frac{2e^3}{n^3}.
    \end{align}
    The remaining argument compares $R$ with the $L^1$-distance $d_1((\bar x,\bar y,\bar z),\mathcal A)$ to the nearest point of $\mathcal A$ and shows that this distance is $O(e/n)$. We provide an overview of the proof here and defer two elementary calculus arguments to Appendix~\ref{app:a3-claims}.

    Letting $\epsilon\colonequals a_0/n$, the sign-normalized Eisenstein inequalities give $\bar y,\bar z \ge 0$ and $|\bar x| \le 1+ \epsilon$ with $\bar y+\bar z-\bar x\le2 + \epsilon$. We let $D_\epsilon$ denote the corresponding domain in $\RR^3$:
    \[
        D_\epsilon\colonequals \{(u,v,w): |u| \le 1+\epsilon, \,
        0 \le v, w \le 1, \, v+w-u \le 2 + \epsilon\}.
    \]
    When $\epsilon$ tends to $0$, this domain becomes the compact convex polytope
    \[
        D\colonequals \{(u,v,w):-1\le u\le1, \; 0\le v,w\le1, v+w-u\le2\}.
    \]

    \begin{figure}[t]
        \centering
        \includegraphics[width=\linewidth]{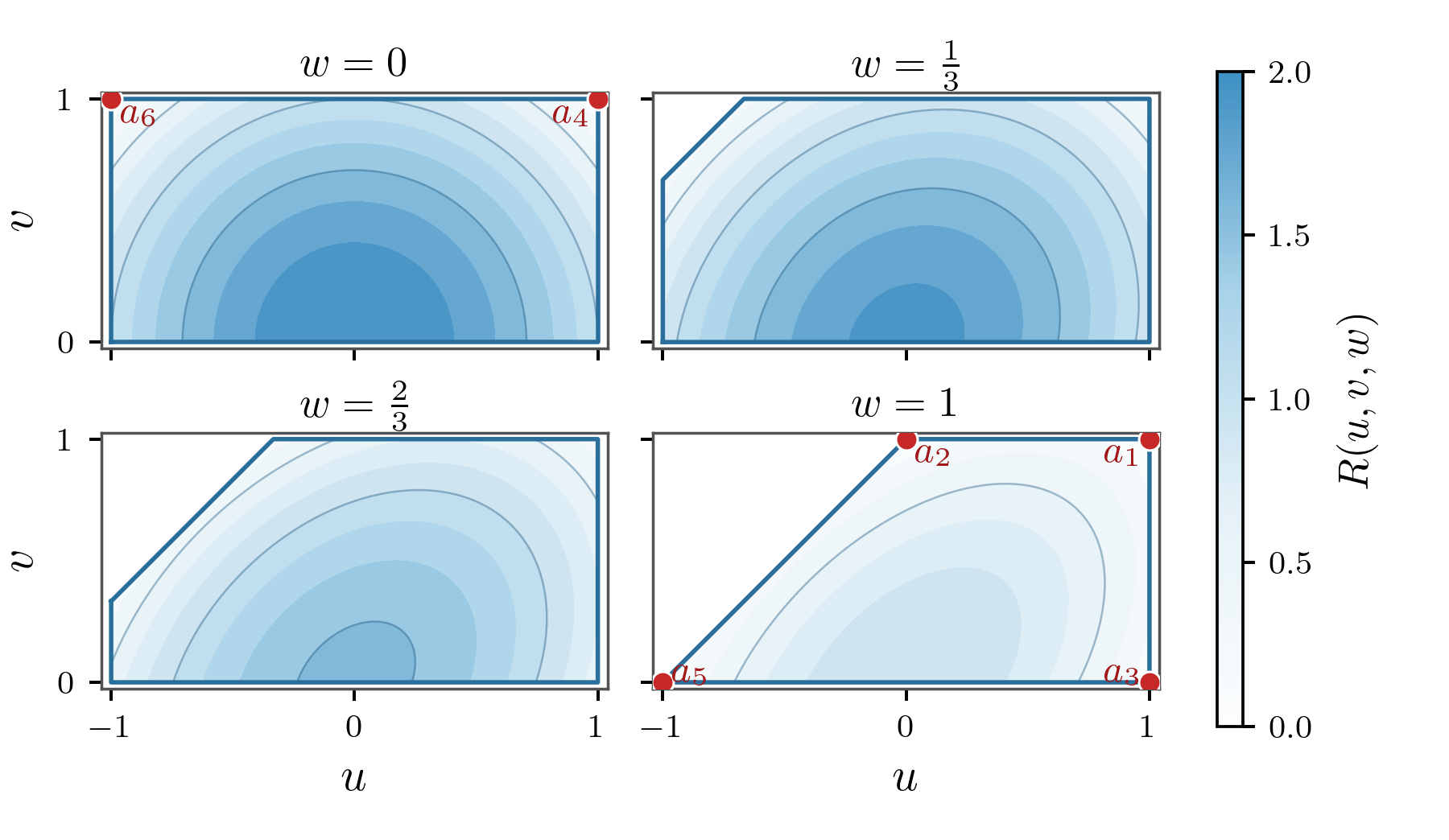}
        \caption{Fixed-$w$ cross-sections of $D$. The thick blue lines mark the boundary of each slice, while the shading and thin contours indicate the value of $R$. The red points $a_1, \dots, a_6$ are the six elements of $\mathcal A$, at which $R=0$.}
        \label{fig:domainofR}
    \end{figure}

    \begin{claim}
        \label{clm:a3-domain-comparison}
        Suppose $0\le\epsilon\le3/50$. For every $(u,v,w)\in D_\epsilon$, there is some $u_D\in\RR$ such that $(u_D,v,w)\in D$ with $|u-u_D|\le\epsilon$ and
        \[
            |R(u_D,v,w)-R(u,v,w)|\le4\epsilon.
        \]
    \end{claim}
    Since $a_0\le 3e/2$ and $n\ge 25e^2$, we have $\epsilon\le 3/50$. Claim~\ref{clm:a3-domain-comparison} therefore gives a point $(x_D,\bar y,\bar z)\in D$ satisfying both $|\bar x-x_D|\le \epsilon$ and $|R(x_D,\bar y,\bar z)-R(\bar x,\bar y,\bar z)|\le 4\epsilon$.

    \begin{claim}
        \label{clm:a3-distance}
        If $(u,v,w)\in D$, then
        \begin{equation}
            \label{eq:a3-distance-claim}
            d_1((u,v,w),\mathcal A)\le4R(u,v,w).
        \end{equation}
    \end{claim}
    Applying both claims along with the triangle inequality gives
    \begin{align*}
        d_1((\bar x,\bar y,\bar z),\mathcal A)
        &\le\epsilon
        +4R(x_D,\bar y,\bar z)\\
        &\le\epsilon
        +4\left(R(\bar x,\bar y,\bar z)+4\epsilon\right).
    \end{align*}
    Substituting back $\epsilon = a_0 / n$ and using \eqref{eq:a3-discriminant-bound} along with $a_0\le3e/2$ and $n\ge25e^2$ yields
    \begin{equation}
        \label{eq:a3-center-distance}
        n\,d_1((\bar x,\bar y,\bar z),\mathcal A)
        <\frac{99}{2}e+\frac{24e^2}{n}+\frac{8e^3}{n^2}
        <51e.
    \end{equation}
    Consequently, there are $\vec c=(c_1,c_2,c_3)\in\mathcal A$ and $\vec k=(k_1,k_2,k_3)\in\ZZ^3$ such that
    \begin{equation}
        \label{eq:a3-pre-chart-bounds}
        (x,y,z)=n\vec c+\vec k,
        \qquad a_0+b_0\le3e,
        \qquad |k_1|+|k_2|+|k_3|<51e.
    \end{equation} 
    From here, we apply a change of variables to move $\vec c\in\mathcal A$ to $(1,1,1)$ in the $Q_0$-chart.
    \begin{claim}
        \label{clm:a3-chart-change}
        The form $Q$ is integrally equivalent to $nQ_0+h$, where
        \[
            h\colonequals aY^2+bZ^2-rYZ-sXZ-tXY
        \]
        for some integers satisfying $0\le a\le b<57e$ and $|r|,|s|,|t|<57e$.
    \end{claim}

    Finally, evaluating $h$ on the six minimal-vector pairs of $Q_0$ recovers the vector $\lambda$ satisfying $h_\lambda=h$. Since $Q_0(\vec v_i)=1$ and $\min(Q)=n$, we have
    \[
        \lambda_i=h(\vec v_i)=Q(\vec v_i)-n\ge0.
    \]
    With the six representatives of $S(Q_0)$ chosen in \eqref{eq:fixedvi}, this gives
    \[
        \lambda=(0,a,b,a+t,b+s,a+b+r)\in\ZZ_{\ge0}^6.
    \]
    In particular, the first coordinate is zero, so $\lambda_{\min}=0$, as desired.
\end{proof}

To refine the lemma, we examine how each perturbation changes the discriminant. By \eqref{eq:hlam-expansion}, the coefficients of $h_\lambda$ are linear in $\lambda$, so $\disc(nQ_0+h_\lambda)$ is homogeneous of degree $3$ in $n,\lambda_1,\ldots,\lambda_6$. Consequently, there are homogeneous polynomials $C_2,C_1,C_0\in\ZZ[\lambda_1,\ldots,\lambda_6]$ of degrees $1$, $2$, and $3$, respectively, such that
\[
    \disc(nQ_0+h_\lambda)=2n^3+C_2(\lambda)n^2+C_1(\lambda)n+C_0(\lambda).
\]
For large $n$, the term $C_2(\lambda)n^2$ is the dominant contribution of $h_\lambda$ to the discriminant. To identify $C_2(\lambda)$, we note $\det(A_{Q_0})=4$, and compute
\begin{align*}
    \disc(nQ_0+h_\lambda)
    &=\frac12\det\left(nA_{Q_0}+A_{h_\lambda}\right)\\
    &=\frac{n^3}{2}\det(A_{Q_0})
    \det\left(I+\frac1nA_{Q_0}^{-1}A_{h_\lambda}\right)\\
    &=2n^3\det\left(I+\frac1nA_{Q_0}^{-1}A_{h_\lambda}\right).
\end{align*}
Using $\det(I+tM)=1+t\operatorname{tr}(M)+O(t^2)$ and the coefficient bounds from \Cref{lem:a3}, we obtain
\[
    \disc(nQ_0+h_\lambda)
    =2n^3+2\operatorname{tr}(A_{Q_0}^{-1}A_{h_\lambda})n^2
    +O(e^2n+e^3).
\]

A direct calculation shows that $2\operatorname{tr}(A_{Q_0}^{-1}A_{h_\lambda})=\lambda_1+\cdots+\lambda_6=\wt(\lambda)$. Therefore, $C_2(\lambda)=\wt(\lambda)$ and we may write
\begin{equation}
    \label{eq:lambda-discriminant-expansion}
    \disc(nQ_0+h_\lambda) = 2n^3+\wt(\lambda)n^2+C_1(\lambda)n+C_0(\lambda).
\end{equation}
This expansion is central to our classification theorem below: for sufficiently large $n$, the weight $\wt(\lambda)$ determines an interval of length comparable to $n^2$ containing $\disc(nQ_0+h_\lambda)$.

\begin{lemma}
    \label{lem:weight-trace}
    Let $\lambda\in\ZZ_{\ge0}^6$ and let $n,d\in\ZZ_{\ge0}$ be such that $Q=nQ_0+h_\lambda$ is positive definite. If $\disc(nQ_0+h_\lambda)>2(n+d)^3$, then $\wt(\lambda)>6d$.
\end{lemma}
\begin{proof}
    The determinant-trace inequality \cite[equation (4), p. 73]{aigner_spectral_2018} gives
    \begin{equation}
        \label{eq:determinant-trace-bound}
        \sqrt[3]{\frac{\disc(Q)}{2}}
        =\sqrt[3]{\det(A_{Q_0}^{-1}A_Q)}
        \le \frac13\operatorname{tr}(A_{Q_0}^{-1}A_Q)
        =n+\frac{\wt(\lambda)}{6},
    \end{equation}
    where the first equality follows from $\det(A_{Q_0})=4$ and
    $\det(A_Q)=2\disc(Q)$, and the last from
    $2\operatorname{tr}(A_{Q_0}^{-1}A_{h_\lambda})=\wt(\lambda)$.
    Moreover, $A_{Q_0}^{-1}A_Q$ is similar to the positive definite symmetric matrix $A_{Q_0}^{-1/2}A_QA_{Q_0}^{-1/2}$, so equality holds in \eqref{eq:determinant-trace-bound} if and only if $A_Q=cA_{Q_0}$ for some $c>0$.
    The hypothesis now yields
    \[
        n+d<\sqrt[3]{\frac{\disc(Q)}2}
        \le n+\frac{\wt(\lambda)}6,
    \]
    so $\wt(\lambda)>6d$ as desired.
\end{proof}

Together, \Cref{lem:a3,lem:weight-trace} and \eqref{eq:lambda-discriminant-expansion} suggest the following refinement. For sufficiently large $n$, every form satisfying
\[
    \min(Q)=n
    \qquad\text{and}\qquad
    2(n+d)^3<\disc(Q)<2(n+d+1)^3
\]
is integrally equivalent to some $nQ_0+h_\lambda$ with $\lambda\in\ZZ_{\ge0}^6$, $\lambda_{\min}=0$, and $\wt(\lambda)\in\{6d+1,\ldots,6d+6\}$.

\Cref{lem:weight-trace} gives the lower bound on $\wt(\lambda)$, while \eqref{eq:lambda-discriminant-expansion} suggests the upper bound. We next show that tuples in this weight range produce forms of minimum $n$. Under a slightly stronger hypothesis, the following lemma also identifies $S(Q_0)$ as the six shortest vector pairs, allowing us to control integral equivalences between these forms.

\begin{lemma}
    \label{lem:wtmin}
    Let $n\in\ZZ_{\ge1}$ and let $\lambda\in\ZZ_{\ge0}^6$ with $\lambda_{\min}=0$. If $\wt(\lambda)\le n$, then $nQ_0+h_\lambda$ is positive definite and $\min(nQ_0+h_\lambda)=n$. Moreover, if $2\wt(\lambda)<n$, the six vector pairs on which $nQ_0+h_\lambda$ takes its smallest values are precisely those in $S(Q_0)$.
\end{lemma}

\begin{proof}
    Recall that the representatives from \eqref{eq:fixedvi} give rise to a basis $\eta_1, \dots, \eta_6$ of integral ternary quadratic forms such that $h_\lambda = \sum \lambda_i \eta_i$.

    We first show that each $Q_0+\eta_i$ is integrally equivalent to $Q_0$ and therefore has minimum $1$. The $S_4$-action described in \eqref{eq:evaluation-equivariance} is transitive on the 6 coordinates $\lambda_i$ and therefore on $\eta_1,\ldots,\eta_6$. Thus, all the forms $Q_0+\eta_i$ are integrally equivalent, and it suffices to consider $Q_0+\eta_4$, with $\eta_4=-XY$. The unimodular substitution $(X,Y,Z)\mapsto(X,-Y,Y+Z)$ gives
    \begin{align*}
        (Q_0+\eta_4)(X,Y,Z)
        &=X^2+Y^2+Z^2+YZ+XZ\\
        &=Q_0(X,-Y,Y+Z).
    \end{align*}
    Hence, $Q_0$ is integrally equivalent to $Q_0+\eta_i$ for every $i$, so $\min(Q_0)=\min(Q_0+\eta_i)=1$. We may now express $nQ_0+h_\lambda$ as a nonnegative combination of these forms. For every $\vec u\ne0$, this gives
    \begin{align*}
        (nQ_0+h_\lambda)(\vec u)
        &=\bigl(n-\wt(\lambda)\bigr)Q_0(\vec u)
        +\sum_i\lambda_i(Q_0+\eta_i)(\vec u)\\
        &\ge\bigl(n-\wt(\lambda)\bigr)
        +\sum_i \lambda_i \ge n.
    \end{align*}
    Thus, $nQ_0+h_\lambda$ is positive definite and has minimum at least $n$. Choosing $j$ with $\lambda_j=0$ gives $(nQ_0+h_\lambda)(\vec v_j)=n$, so its minimum is exactly $n$.

    If additionally we have $2\wt(\lambda)<n$ and $Q_0(\vec u)\ge2$, the same decomposition gives
    \[
        (nQ_0+h_\lambda)(\vec u) \ge 2\bigl(n-\wt(\lambda)\bigr)+\wt(\lambda)
        =2n-\wt(\lambda)>n+\wt(\lambda).
    \]
    For each representative $\vec v_i$ chosen in \eqref{eq:fixedvi}, however,
    \[
        (nQ_0+h_\lambda)(\vec v_i)=n+\lambda_i\le n+\wt(\lambda).
    \]
    Every $\vec u$ representing a vector pair outside $S(Q_0)$ satisfies $Q_0(\vec u)\ge2$, so these inequalities prove the additional conclusion.
\end{proof}

The preceding lemma shows that tuples in the relevant weight range produce forms of minimum $n$. For sufficiently large $n$, it also shows that the six vector pairs defining the $\lambda$-coordinates can be recovered from the form itself. The following theorem completes the correspondence between the weight of $\lambda$ and the discriminant of $nQ_0+h_\lambda$. It also shows that integral equivalence of the associated forms corresponds exactly to the $S_4$-action described in \eqref{eq:evaluation-equivariance}.

\begin{theorem}
    \label{thm:wt-param}
    Let $n, d\in\ZZ$ with $d\ge0$ and $n\ge3000(d+1)^2$, and define
    \[
        \Lambda_d\colonequals
        \left\{
            \lambda\in\ZZ_{\ge0}^6:
            \lambda_{\min} = 0, \quad
            6d<\wt(\lambda) \le 6(d+1)
        \right\}.
    \]
    The map $\lambda\longmapsto nQ_0+h_\lambda$ is a bijection between the $S_4$-orbits of $\Lambda_d$ and equivalence classes of positive definite integral ternary quadratic forms $Q$ satisfying
    \[
        \min(Q)=n
        \quad\text{ and }\quad
        2(n+d)^3<\disc(Q)<2(n+d+1)^3.
    \]
\end{theorem}
\begin{proof}
    Let $e=d+1$. First, suppose that $Q$ is a positive definite integral ternary quadratic form with minimum $n$ and discriminant in the stated range. Since $n\ge3000e^2\ge25e^2$, \Cref{lem:a3} gives an equivalent representative $Q_\lambda=nQ_0+h_\lambda$ with $\lambda\in\ZZ_{\ge0}^6$ and $\lambda_{\min}=0$. The lower discriminant bound and \Cref{lem:weight-trace}
    then give $\wt(\lambda)>6d$.

    To obtain the upper bound on the weight, we compare \eqref{eq:lambda-discriminant-expansion} with the upper bound on the discriminant:
    \begin{equation}
        \label{eq:upper-weight-comparison}
        0<2(n+e)^3-\disc(Q_\lambda)
        =(6e-\wt(\lambda))n^2+(6e^2-C_1(\lambda))n
        +2e^3-C_0(\lambda).
    \end{equation}
    Since the discriminant is preserved by the change of variables in the proof of \Cref{lem:a3}, we compute $C_1(\lambda)$ and $C_0(\lambda)$ in the pre-$Q_0$-chart, where sharper coefficient bounds are available. In the notation there, expanding the discriminant gives
    \begin{equation}
        \label{eq:c1-c0-formulas}
        \begin{aligned}
            C_1(\lambda)
            &=4a_0b_0+c_1k_2k_3+c_2k_1k_3+c_3k_1k_2
            -k_1^2-k_2^2-k_3^2\\
            &\quad-2a_0c_2k_2-2b_0c_3k_3,\\
            C_0(\lambda)
            &=k_1k_2k_3-a_0k_2^2-b_0k_3^2
        \end{aligned}
    \end{equation}
    for some $\vec c\in\mathcal A$. The bounds in \eqref{eq:a3-pre-chart-bounds}, together with $0\le a_0,b_0\le3e$ and $|c_i|\le1$, give $|C_1(\lambda)|<2916e^2$ and $|C_0(\lambda)|<12716e^3$. Similarly, if $\wt(\lambda)>6e$, then $\wt(\lambda)\ge6e+1$, and the coefficient bounds above together with \eqref{eq:upper-weight-comparison} imply
    \[
        0<-n^2+2922e^2n+12718e^3.
    \]
    The right-hand side is negative when $n\ge2927e^2$, contradicting the hypothesis $n\ge3000e^2$. Thus, $\wt(\lambda)\le6e$ and $\lambda\in\Lambda_d$, proving that every equivalence class in the stated range is represented by some $Q_\lambda$ with $\lambda\in\Lambda_d$.

    For the reverse direction, let $\lambda\in\Lambda_d$ with $Q_\lambda\colonequals nQ_0+h_\lambda$. Since $\wt(\lambda)\le6e\le n$ and $\lambda_{\min}=0$, \Cref{lem:wtmin} gives $\min(Q_\lambda)=n$.
    Moreover, the contrapositive of \Cref{lem:weight-trace} gives $\disc(Q_\lambda)\le2(n+e)^3$. If equality held, then
    \[
        n+e=\sqrt[3]{\frac{\disc(Q_\lambda)}2}
        \le n+\frac{\wt(\lambda)}6\le n+e,
    \]
    which forces equality in both inequalities above. Equality in the second gives $\wt(\lambda)=6e$, while equality in the first is the equality case of the determinant-trace bound \eqref{eq:determinant-trace-bound}. The latter gives $Q_\lambda=(n+e)Q_0$, and hence $\lambda_i=e$ for every $i$, contradicting $\lambda_{\min}=0$. Thus, we have $\disc(Q_\lambda)< 2(n+e)^3$.

    For the lower bound, we similarly compare \eqref{eq:lambda-discriminant-expansion} with $2(n+d)^3$. Since $\wt(\lambda)\ge6d+1$, we have
    \begin{align*}
        \disc(Q_\lambda)-2(n+d)^3
        &=(\wt(\lambda)-6d)n^2
        +(C_1(\lambda)-6d^2)n+C_0(\lambda)-2d^3 \\
        &\ge n^2+(C_1(\lambda)-6d^2)n+C_0(\lambda)-2d^3.
    \end{align*}
    Here the weight bound allows us instead to work directly in the $Q_0$-chart. Write
    \[
        h_\lambda=aX^2+bY^2+cZ^2+rYZ+sXZ+tXY.
    \]
    By \eqref{eq:hlam-expansion}, each of these coefficients has absolute value at most $\wt(\lambda)\le6e$. Expanding $\disc(nQ_0+h_\lambda)$ in these coefficients gives
    \begin{align*}
        C_1(\lambda)
        &=4(ab+ac+bc)+rs+rt+st-r^2-s^2-t^2
        -2ar-2bs-2ct,\\
        C_0(\lambda)
        &=4abc+rst-ar^2-bs^2-ct^2.
    \end{align*}
    Counting terms with multiplicity, we find that the formulas for $C_1(\lambda)$ and $C_0(\lambda)$ contain $24$ terms of degree $2$ and $8$ terms of degree $3$, respectively. Hence, $|C_1(\lambda)|\le24(6e)^2$ and $|C_0(\lambda)|\le8(6e)^3$.
    Substituting these bounds gives
    \[
        \disc(Q_\lambda)-2(n+d)^3 \ge n^2-870e^2n-1730e^3.
    \]
    The right-hand side is positive for every $n\ge3000e^2$, and hence $\disc(Q_\lambda)>2(n+d)^3$.

    It remains to show that $Q_\lambda$ and $Q_{\lambda'}$ are integrally equivalent exactly when $\lambda$ and $\lambda'$ lie in the same $S_4$-orbit. By the equivariance in \eqref{eq:evaluation-equivariance}, tuples in the same $S_4$-orbit give integrally equivalent forms. 
    
    Conversely, suppose that $Q_\lambda$ and $Q_{\lambda'}$ are integrally equivalent for $\lambda,\lambda'\in\Lambda_d$, and choose $U\in\operatorname{GL}_3(\ZZ)$ such that $Q_\lambda=Q_{\lambda'}\circ U$. Since $2\wt(\lambda),2\wt(\lambda')\le12e<n$, the additional conclusion of \Cref{lem:wtmin} shows that $U$ permutes the six pairs in $S(Q_0)$. The form $Q_0$ is uniquely determined by taking the value $1$ on these pairs, so $U\in\operatorname{Aut}(Q_0)$. Hence, its permutation of $S(Q_0)$ comes from the natural $S_4$-action. If $U\vec v_i=\pm\vec v_j$, then
    \[
        \lambda_i=Q_\lambda(\vec v_i)-n
        =Q_{\lambda'}(U\vec v_i)-n
        =\lambda'_j.
    \]
    Thus, the corresponding element of $S_4$ carries $\lambda$ to $\lambda'$, so they lie in the same orbit.
\end{proof}

The constant $3000$ is chosen to keep the hypothesis simple and is not intended to be optimal. For prime discriminants, it can be replaced by a substantially smaller, computable bound; see \Cref{rem:prime-discriminant-bound}.

We investigate two consequences of the parameterization in \Cref{thm:wt-param}. First, in \Cref{sec:wisde-primes}, we use it to give an algorithm that detects and constructs supersingular elliptic curves $E$ with near-maximal $\delta_E$. Second, in \Cref{sec:wisde-deficits}, we use it to count the primes $p$ for which $\delta(p)$ is near-maximal.

\section{Parameterizing Wisde Primes by Cubics}
\label{sec:wisde-primes}

Fix $d\ge0$ and, for each prime $p$, set
\[
    n \colonequals \left\lfloor\sqrt[3]{p/2}\right\rfloor-d.
\]
Combining \Cref{thm:wt-param,prop:curve_tqf} gives the following cubic characterization of the primes $p$ for which some supersingular $E/\overline{\mathbb F}_p$ satisfies $\delta_E=n$.

\begin{theorem}
    \label{thm:fixed-distance-cubics}
    For every $d\in\ZZ_{\ge0}$, there exists a finite set $\mathcal P_d\subset\ZZ[x]$ of cubic polynomials such that, for every prime $p\ge 2\left(3000(d+1)^2+d\right)^3$, there exists a supersingular elliptic curve $E/\overline{\mathbb F}_p$ satisfying
    \[
        \delta_E=\left\lfloor\sqrt[3]{p/2}\right\rfloor-d
    \]
    if and only if
    \[
        p=P\left(\left\lfloor\sqrt[3]{p/2}\right\rfloor-d\right)
    \]
    for some $P\in\mathcal P_d$.
    Moreover, $\#\mathcal P_d=O((d+1)^4)$, and $\mathcal P_d$ can be constructed in $O((d+1)^4\log^2(d+1))$ bit operations.
\end{theorem}

\begin{proof}
    Let $n\colonequals\lfloor\sqrt[3]{p/2}\rfloor-d$. The assumption on $p$ ensures that $n\ge3000(d+1)^2$; moreover, the definition of $n$ and the floor function give
    \[
        2(n+d)^3<p<2(n+d+1)^3,
    \]
    with the first inequality strict because $p$ is prime and $n+d>1$.

    By \Cref{thm:wt-param}, the equivalence classes of positive definite integral ternary quadratic forms of minimum $n$ and discriminant in the interval above are parameterized by $nQ_0+h_\lambda$ as $\lambda$ ranges over representatives of $\Lambda_d/S_4$. For each such $\lambda$, define
    \[
        P_\lambda(x)\colonequals\disc(xQ_0+h_\lambda).
    \]
    Each $P_\lambda(x)$ is a cubic polynomial with integer coefficients depending only on $\lambda$.
    Moreover, since integral equivalence preserves discriminants, $P_\lambda$ depends only on the orbit $[\lambda]\in\Lambda_d/S_4$, so we may define
    \[
        \mathcal P_d\colonequals\{P_\lambda:[\lambda]\in\Lambda_d/S_4\}.
    \]
    Finally, by \Cref{prop:curve_tqf}, a supersingular $E/\overline{\mathbb F}_p$ with $\delta_E=n$ exists if and only if $p$ is the discriminant of one of these forms, equivalently, if and only if $p=P_\lambda(n)$ for some $P_\lambda\in\mathcal P_d$.

    We now describe an algorithm for constructing representatives of $\Lambda_d/S_4$ from \emph{weak compositions}. That is, for each $w\in\{6d+1,\ldots,6d+6\}$, we enumerate the tuples
    \[
        (\mu_1,\ldots,\mu_5)\in\ZZ_{\ge0}^5,
        \qquad
        \mu_1+\cdots+\mu_5=w.
    \]
    Inserting a zero in each of the six possible positions produces every element of $\Lambda_d$, possibly with repetition.
    For each $w$, there are $\binom{w+4}{4}$ such compositions, and so we have
    \[
        \#\Lambda_d
        \le 6\sum_{w=6d+1}^{6d+6}\binom{w+4}{4}
        \le 36\binom{6d+10}{4}
        =O((d+1)^4).
    \]
    We then deduplicate these tuples under the $S_4$-action to obtain one representative of each orbit and compute the corresponding polynomials $P_\lambda$. Since $|S_4|=24$, this does not change the asymptotic bound, and hence $\#\mathcal P_d=O((d+1)^4)$.

    Finally, the explicit formula in \eqref{eq:lambda-discriminant-expansion} shows that each $P_\lambda$ can be computed using a constant number of additions and multiplications, independent of $d$, on integers of bit length $O(\log(d+1))$. If multiplication of such integers is taken to require $O(\log^2(d+1))$ bit operations, then the construction of $\mathcal P_d$ requires $O((d+1)^4\log^2(d+1))$ bit operations.
\end{proof}

Applying the theorem above simultaneously for $0\le d\le D$ gives the following algorithm for detecting when $\delta(p)$ lies within $D$ of its upper bound.

\begin{corollary}
    \label{cor:deficit-alg}
    For every $D\in\ZZ_{\ge0}$, there exists an algorithm which, given a prime $p\ge 2(3000(D+1)^2+D)^3$, runs in $O((D+1)^5\log^2 p)$ bit operations and either computes $\delta(p)$ or certifies that
    \[
        \delta(p)<\left\lfloor\sqrt[3]{p/2}\right\rfloor-D.
    \]
    In the first case, assuming GRH, the algorithm can also return a supersingular elliptic curve $E/\mathbb F_{p^2}$ satisfying $\delta_E=\delta(p)$ in probabilistic polynomial time in $\log p$.
\end{corollary}
\begin{proof}
    Let $m\colonequals\lfloor\sqrt[3]{p/2}\rfloor$. For each $0\le d\le D$, construct $\mathcal P_d$ and test whether
    \[
        p=P(m-d)
    \]
    for some $P\in\mathcal P_d$. The lower bound on $p$ allows us to apply \Cref{thm:fixed-distance-cubics} for every $d\le D$. We perform these tests in increasing order of $d$ so that the first match occurs at the smallest $d$ and the resulting curve $E$ with $\delta_E=m-d$ is maximal in characteristic $p$. If no match occurs, then we know $\delta(p)<m-D$.

    Computing $\mathcal P_d$ for each $0\le d\le D$ produces $O((D+1)^5)$ polynomials in total and requires $O((D+1)^5\log^2(D+1))$ bit operations. Each evaluation above requires a constant number of additions and multiplications on integers of bit length $O(\log p)$. If multiplication is taken to require $O(\log^2 p)$ bit operations, then the construction and all evaluations require $O((D+1)^5\log^2 p)$ bit operations in total.

    Finally, when constructing $\mathcal P_d$, we may retain one representative $\lambda$ for each polynomial $P_\lambda$. A match $p=P_\lambda(m-d)$ then recovers the ternary quadratic form
    \[
        Q=(m-d)Q_0+h_\lambda.
    \]
    Under GRH, we can use the effective Deuring correspondence outlined in \Cref{rem:curve_tqf_effective} to convert the form $Q$ to a supersingular elliptic curve $E/\mathbb F_{p^2}$ with $\delta_E=\min(Q)=\delta(p)$ in probabilistic polynomial time in $\log p$.
\end{proof}

\begin{remark}
    \label{rem:prime-discriminant-bound}
    The lower bounds in \Cref{thm:fixed-distance-cubics,cor:deficit-alg} are inherited from the uniform bound in \Cref{thm:wt-param}, which applies to all discriminants, and their numerical constants are not intended to be sharp. Appendix~\ref{app:fixed-deficit-bound} gives a finite procedure that, for each fixed $d$, determines the least $n$ from which $\mathcal P_d$ detects every relevant prime discriminant. Applying these results simultaneously for $0\le d\le D$ replaces the large constant in the $O(D^6)$ size condition of \Cref{cor:deficit-alg} by a computable and sharper one.
\end{remark}

We provide an implementation of this algorithm in the computational supplement \cite{swanson_computational_2026}. With the extension to $n\ge27$ established in Appendix~\ref{app:d0-computation}, taking $D=0$ independently reproduces the primes in \cite[Table~5.1]{aubry_minimal_2026} that attain the general upper bound on $\delta(p)$. Aubry, Oyono, and Vincent compute these primes by enumerating the maximal-order types in $B_{p,\infty}$ and determining $\delta_E$ for each corresponding curve class, whereas our algorithm computes them efficiently from the cubic family $\mathcal P_0$.

We now consider the case $D=0$, where the family $\mathcal P_0$ gives an explicit classification of the wisde primes. After filtering out cubics that clearly cannot take on prime values for $n\ge27$, we obtain the following result.

\begin{corollary}
    \label{thm:wisde-classification}
    Let $n \ge 27$ and let $2n^3 \le p < 2(n + 1)^3$ be prime. Then $p$ is wisde if and only if $p = P_i(n)$ for some $1\le i\le9$, where
    \begin{equation}
        \label{eq:cubic-families}
        \begin{array}{@{}ll@{}}
            P_1(x)=2x^3+4x^2-3x-4,
            &P_2(x)=2x^3+5x^2+x-1,\\
            P_3(x)=2x^3+5x^2-7x-9,
            &P_4(x)=2x^3+6x^2+x-4,\\
            P_5(x)=2x^3+6x^2-9x-24,
            &P_6(x)=2x^3+6x^2-x-1,\\
            P_7(x)=2x^3+6x^2+x-2,
            &P_8(x)=2x^3+6x^2-13x-16,\\
            P_9(x)=2x^3+6x^2-3x-6.
        \end{array}
    \end{equation}
\end{corollary}

\begin{proof}
    Suppose first that $p$ is wisde. For $n\ge3000$, \Cref{thm:fixed-distance-cubics} with $d=0$ yields $70$ orbits in $\Lambda_0/S_4$ and $69$ distinct cubic polynomials $P_\lambda$. Of these cubics, $37$ have a nontrivial fixed divisor: some integer $D>1$ divides $P_\lambda(m)$ for every $m\in\ZZ$.
    Another $23$ cubics factor over $\ZZ$, and evaluating their factors shows that they take no prime values for $n\ge27$.
    The remaining 9 cubics are the $P_1,\ldots,P_9$ above. Thus, $p=P_i(n)$ for some $i$ whenever $n\ge3000$, and \Cref{thm:fixed-distance-cubics} gives the converse in this range. Appendix~\ref{app:d0-computation} gives a complete enumeration of the remaining perturbations and shows that the only prime discriminant not detected by $P_1,\ldots,P_9$ occurs at $n=26$. It also exhibits, for each $P_i$, a form of minimum $n$ and discriminant $P_i(n)$, proving the converse for every $n\ge27$.
\end{proof}

\begin{figure}[ht]
    \centering
    \includegraphics[width=\textwidth]{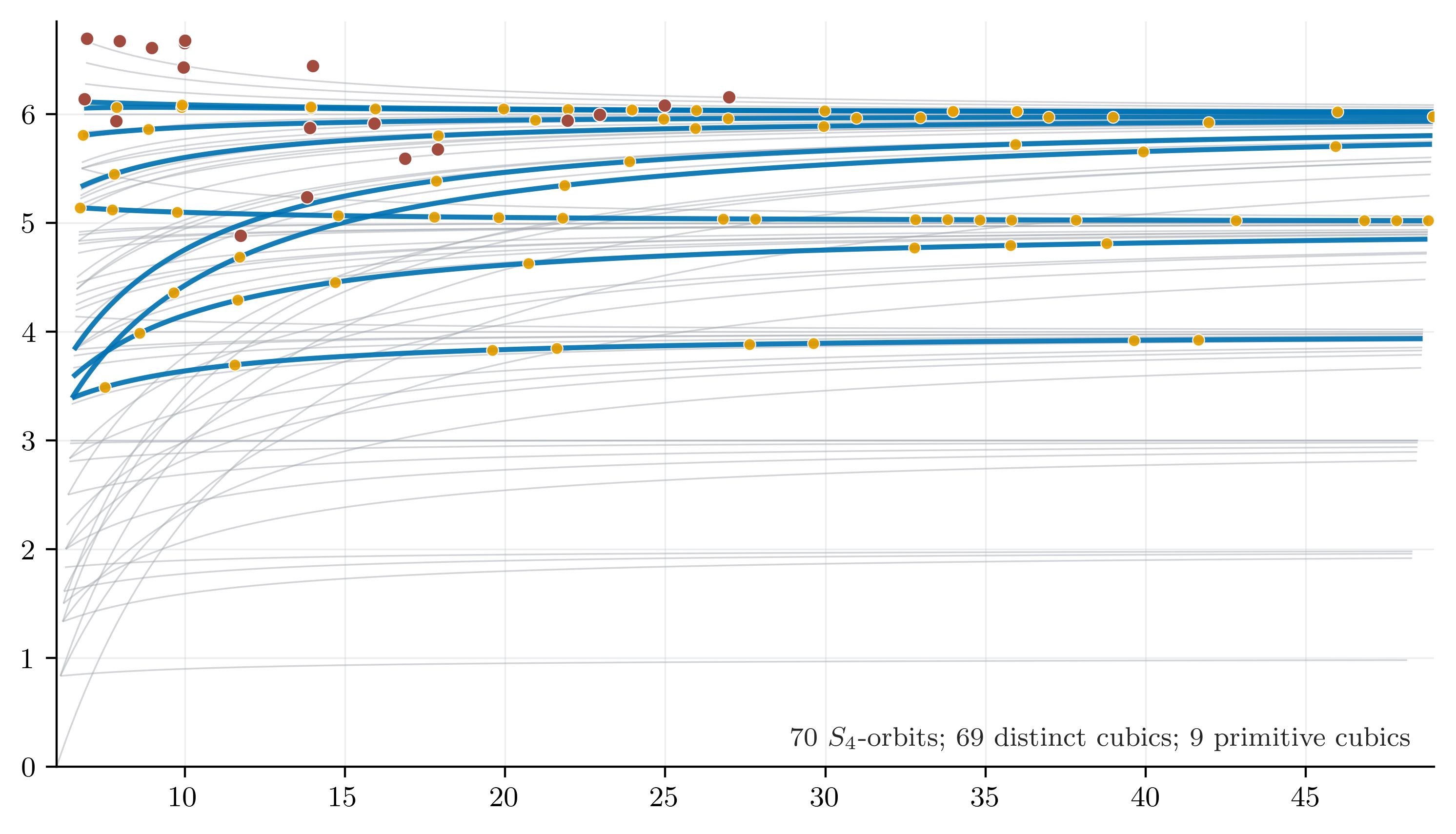}
    \caption{The $d=0$ cubics $P_\lambda\in\mathcal P_0$ from \Cref{thm:fixed-distance-cubics}. We plot $(P_\lambda(n)-2n^3)/n^2$ against $n$ for $2\le n\le10$. The faded cubics take no prime values for $n\ge27$; the nine blue cubics are those in \Cref{thm:wisde-classification}.}
    \label{fig:cubic-family-phase}
\end{figure}

The classification immediately gives an unconditional upper bound, showing that wisde primes are asymptotically negligible among all primes.

\begin{corollary}
    \label{cor:unconditional-wisde-count}
    For any $X \ge 2$,
    \[
        \#\{p\le X:p\text{ is wisde}\}
        \le 9(X/2)^{1/3}+O(1).
    \]
\end{corollary}
\begin{proof}
    All but finitely many wisde primes are covered by \Cref{thm:wisde-classification}. For each such $p\le X$, its integer $n=\lfloor(p/2)^{1/3}\rfloor$ is at most $(X/2)^{1/3}$, and there are at most nine values $P_i(n)$ for each $n$.
\end{proof}

The unconditional bound above still leaves open whether there are infinitely many wisde primes. By \Cref{thm:wisde-classification}, this becomes an instance of the classical problem of determining whether a polynomial in $\ZZ[x]$ takes prime values infinitely often \cite{bouniakowsky_diviseurs_1857}.
Even for an irreducible polynomial $P(x)\in\ZZ[x]$ of degree at least two with no fixed divisor, it remains unknown whether $P(n)$ is prime for infinitely many $n\in\ZZ_{\ge0}$. The single-polynomial Bateman--Horn conjecture predicts an affirmative answer and, more precisely, the following asymptotic.

\begin{conjecture}[Bateman--Horn \cite{bateman_heuristic_1962}]
    \label{conj:bateman-horn}
    Let $P(x)\in\ZZ[x]$ be irreducible with positive leading coefficient and no fixed divisor. Then\footnote{Here $f(N)\sim g(N)$ means that $f(N)/g(N)\to1$ as $N\to\infty$.}
    \[
        \#\{1\le n\le N:P(n)\text{ is prime}\}
        \sim C(P)\frac{N}{\deg(P)\log N},
    \]
    where
    \[
        C(P)=\prod_{\ell\ \mathrm{prime}}
        \frac{1-\rho_P(\ell)/\ell}{1-1/\ell},
        \quad
        \rho_P(\ell)=\#\{a\bmod\ell:P(a)\equiv0\pmod\ell\}.
    \]
\end{conjecture}

While we refer to the single-polynomial version above simply as the Bateman--Horn conjecture, the general conjecture concerns integers at which a collection of polynomials simultaneously takes prime values. It implies the twin-prime conjecture, the infinitude of Sophie Germain primes, and many other open problems in number theory. See \cite{aletheia-zomlefer_batemanhorn_2020} for an exposition.

\begin{corollary}
    \label{cor:bateman-horn-wisde-primes}
    Assume the Bateman--Horn conjecture. Then
    \[
        \#\{p\le X:p\text{ is wisde}\} \sim C\frac{X^{1/3}}{\log X},
        \quad \text{ where } \quad C=2^{-1/3}\sum_{i=1}^{9}C(P_i) \approx 14.12.
    \]
    In particular, there are infinitely many wisde primes.
\end{corollary}

\begin{proof}
    Since $P_i(n)=2n^3+O(n^2)$, the Bateman--Horn conjecture gives, for each $1\le i\le9$,
    \[
        \#\{n\ge27:P_i(n)\le X,\ P_i(n)\text{ prime}\}
        \sim 2^{-1/3}C(P_i)\frac{X^{1/3}}{\log X}.
    \]
    Moreover, one can check that the cubics in \eqref{eq:cubic-families} satisfy
    \[
        \begin{aligned}
            2n^3&<P_1(n)<P_3(n)<P_2(n)<P_8(n)<P_5(n),\\
            P_5(n)&<P_9(n)<P_6(n)<P_4(n)<P_7(n)<2(n+1)^3,
        \end{aligned}
    \]
    for every $n\ge27$. In particular, the integers $P_1(n),\ldots,P_9(n)$ are pairwise distinct in $[2n^3,2(n+1)^3)$. By \Cref{thm:wisde-classification}, their prime values account for all but finitely many wisde primes, so summing over $i$ proves the result.
\end{proof}

\begin{remark}
    Note that $P_7(x)=P_4(x)+2$ and that $P_4(x)P_7(x)$ has no fixed divisor. Thus, under the two-polynomial Bateman--Horn conjecture, there are infinitely many \emph{twin wisde primes}. For example, $n=33$ gives the pair $(78437,78439)$. We find this fact remarkably unhelpful and equally exciting.
\end{remark}

Thus, the Bateman--Horn conjecture implies the infinitude of wisde primes conjectured in \cite[Conjecture~6.6]{aubry_minimal_2026}, but proving this unconditionally would require a major breakthrough in number theory: showing that an irreducible cubic polynomial without a fixed divisor takes prime values infinitely often. Viewed as the $D=0$ case of \Cref{cor:deficit-alg}, this is part of a more general barrier: for any fixed $D$, determining how often the algorithm computes $\delta(p)$ requires understanding the prime values of finitely many cubics.

When $D$ is allowed to grow with $p$, however, it is natural to ask how often this algorithm actually computes $\delta(p)$, rather than certifying that $\delta(p)$ is not within $D$ of its upper bound. In the next section, we show that even when $D$ grows nearly as quickly as the size condition allows, the algorithm computes $\delta(p)$ for only an asymptotically negligible proportion of primes.

\section{Prime Characteristics Admitting Near-Maximal Frobenius Distance}
\label{sec:wisde-deficits}

In this section, we are concerned with the gap between $\delta(p)$ and its upper bound $\lfloor \sqrt[3]{p/2}\rfloor$. For each prime $p$, we define
\[
    \Delta(p)\colonequals \left \lfloor\sqrt[3]{\frac p2} \right \rfloor-\delta(p)
\]
and study the distribution of primes $p$ for which $\Delta(p)$ is small using \Cref{thm:wt-param}. Aubry, Oyono, and Vincent conjecture that $\Delta(p)=O(\log p)$ in \cite[Conjecture~6.8]{aubry_minimal_2026}.\footnote{They compute $\Delta(p)$ exhaustively for $p\le22,000$ and $93,312\le p\le101,306$; elsewhere up to $265,207$, a sieve only classifies wisde primes \cite[Sections~5.3 and~6.2]{aubry_minimal_2026}. On these ranges, $\log p$ and $p^{1/6}$ are not meaningfully distinguishable.} We show that the typical gap is substantially larger: for almost all primes $p$,
\[
    \Delta(p)>p^{1/6}(\log p)^{-1/2-\varepsilon}
\]
for every $\varepsilon>0$. In particular, $\Delta(p)\ge p^{1/6-o(1)}$ for almost all primes.

The proof is a counting argument based on \Cref{thm:fixed-distance-cubics}. For fixed $D\ge1$, the union
\[
    \bigcup_{0\le d\le D}\mathcal P_d
\]
contains $O(D^5)$ cubics, and all but finitely many primes $p\le X$ satisfying $\Delta(p)\le D$ occur among their values in $[1,X]$. Each cubic takes $O(X^{1/3})$ integer values in this interval, so these primes lie in a set of size $O(X^{1/3}D^5)$. Taking $D=X^{2/15-\varepsilon}$ for any fixed $\varepsilon>0$ gives
\[
    X^{1/3}D^5=X^{1-5\varepsilon}=o(\pi(X)),
\]
where $\pi(X)$ denotes the prime-counting function.
This yields the preliminary bound $\Delta(p)\ge p^{2/15-o(1)}$ for almost all primes $p$. However, this argument ignores the additional concentration imposed by the parameterization in \Cref{thm:wt-param}.

\begin{figure}[ht]
    \centering
    \includegraphics[width=\linewidth]{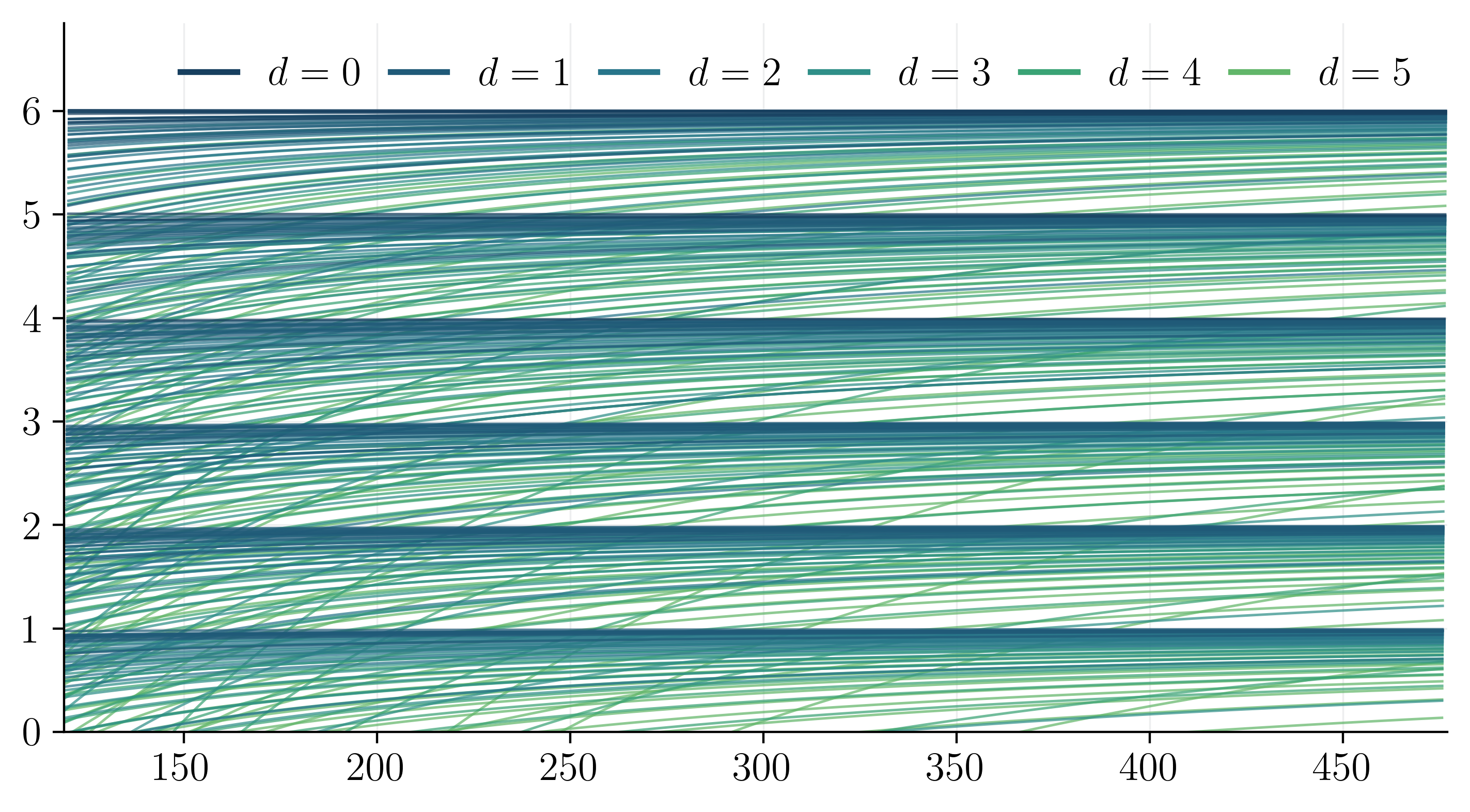}
    \caption{A representative sample of the irreducible, primitive cubics $P_\lambda\in\mathcal P_d$ for $0\le d\le5$. We plot $(P_\lambda(x)-2(x+d)^3)/(x+d)^2$ against $x+d$ as in \Cref{fig:cubic-family-phase}, with color indicating $d$.}
    \label{fig:cumulative-deficit-cubics}
\end{figure}

The concentration visible in \Cref{fig:cumulative-deficit-cubics} is what improves this count. For each $m\ge0$, consider the interval
\[
    I_m\colonequals[2m^3,2(m+1)^3).
\]
As these consecutive intervals range up to $X$, \Cref{thm:fixed-distance-cubics} shows that every relevant prime in $I_m$ arises as $P_\lambda(m-d)$ for some $d\le D$ and $P_\lambda\in\mathcal P_d$. The quadratic coefficient of $P_\lambda$ is $\wt(\lambda)=6d+k$, where $1\le k\le6$, and evaluation at $m-d$ cancels the $6d$-term at quadratic order. As the proof below makes precise, the relevant primes in $I_m$ are therefore confined to six subintervals of total length $O(m(D+1)^2)$.

\begin{theorem}
    \label{thm:deficit-count}
    For all $X, D\in\ZZ_{\ge0}$,
    \[
        \#\{p\le X:p\text{ is prime and }\Delta(p)\le D\}
        =O\left(X^{2/3}(D+1)^2\right),
    \]
    where the implied constant is independent of $X$ and $D$.
\end{theorem}

\begin{proof}
    Fix $p\le X$ and define $d\colonequals\Delta(p)\le D$ as well as $n\colonequals\delta(p)$. Suppose first that $n\ge3000(d+1)^2$. By \Cref{thm:fixed-distance-cubics}, we then have $p=P_\lambda(n)$ for some $P_\lambda\in\mathcal P_d$. Since $\lambda\in\Lambda_d$, each $\lambda_i$ is $O(d+1)$, and the expansion in \eqref{eq:lambda-discriminant-expansion} gives
    \[
        p = P_\lambda(n) =
        2n^3+\wt(\lambda)n^2 +O\bigl(n(d+1)^2+(d+1)^3\bigr),
    \]
    where $\wt(\lambda)=6d+k$ for some $1\le k\le6$.
    Writing this cubic in terms of $m\colonequals n+d=\lfloor\sqrt[3]{p/2}\rfloor$ yields
    \[
        p=2m^3+km^2+O\left(m(d+1)^2\right),
    \]
    where we use $m\ge d+1$ to absorb the terms of order $(d+1)^3$. Since $d\le D$, for each of the six possible values of $k$, every such prime lies within $O(m(D+1)^2)$ of $2m^3+km^2$.
    Thus, for each fixed $m$, these primes lie in six subintervals of $I_m$, each of length $O(m(D+1)^2)$.

    Now let $I_m$ range over the consecutive intervals up to $X$. Since $p\le X$ implies $m\le\lfloor(X/2)^{1/3}\rfloor$, summing this bound over $m$ gives
    \begin{align*}
        O\left((D+1)^2
        \sum_{0\le m\le(X/2)^{1/3}}m\right)
        &=O\left(X^{2/3}(D+1)^2\right).
    \end{align*}

    It remains to account for the primes too small for
    \Cref{thm:fixed-distance-cubics} to apply. If $n<3000(d+1)^2$, then $\lfloor\sqrt[3]{p/2}\rfloor=n+d=O((D+1)^2)$, and consequently $p=O((D+1)^6)$. There are therefore $O(\min\{X,(D+1)^6\})$ such primes. Since $\min\{a,b\}\le a^{2/3}b^{1/3}$ for all $a,b\ge0$, this is $O(X^{2/3}(D+1)^2)$, as required.
\end{proof}

\begin{corollary}
    \label{cor:almost-all-deficits}
    For every $\varepsilon>0$,
    \[
        \lim_{X\to\infty}
        \frac{\#\{p\le X:p\text{ is prime and }
        \Delta(p)>p^{1/6}(\log p)^{-1/2-\varepsilon}\}}{\pi(X)}
        =1.
    \]
    In particular, for almost all primes $p$,
    \[
        \Delta(p)\ge p^{1/6-o(1)}.
    \]
\end{corollary}

\begin{proof}
    Fix $\varepsilon>0$ and consider the exceptional primes $p\le X$ satisfying $\Delta(p)\le p^{1/6}(\log p)^{-1/2-\varepsilon}$. For sufficiently large $X$, taking
    \[
        D=\left\lfloor X^{1/6}(\log X)^{-1/2-\varepsilon}\right\rfloor
    \]
    ensures that every exceptional prime satisfies $\Delta(p)\le D$. By \Cref{thm:deficit-count}, the number of exceptional primes is
    \[
        O\left(X^{2/3}(D+1)^2\right)
        =O\left(\frac{X}{(\log X)^{1+2\varepsilon}}\right).
    \]
    Since $\pi(X)\sim X/\log X$, their proportion among the primes up to $X$ is $O((\log X)^{-2\varepsilon})=o(1)$, as desired.
\end{proof}

The lower bound in \Cref{cor:almost-all-deficits} raises the complementary question of whether, for almost all primes $p$,
\[
    \Delta(p)\le p^{1/6+o(1)}.
\]
Although we are far from proving this upper bound, the following heuristic suggests that it may hold.

Recall the intervals
\[
    I_m=[2m^3,2(m+1)^3),
\]
which have length $6m^2+6m+2$. For $D\ge1$ and $m$ sufficiently large relative to $D^2$, the proof of \Cref{thm:deficit-count} confines the primes $p\in I_m$ satisfying $\Delta(p)\le D$ to six subintervals of total length $O(mD^2)$. At $D \approx m^{1/2}$, we know this bound becomes $O(m^2)$, the same order as the length of $I_m$. If $p$ is chosen uniformly from the primes in $I_m$, and if the union of these subintervals has the same average prime density as $I_m$, then we expect
\[
    \Pr\bigl[\Delta(p)\le D\bigr]
    =O\left(\frac{D^2}{m}\right).
\]

This estimate is only a heuristic upper bound for the proportion of primes satisfying $\Delta(p)\le D$; it gives no upper bound on $\Delta(p)$, since a prime in one of the six subintervals need not equal any candidate value $P_\lambda(m-d)$. A converse prediction requires two further heuristic assumptions. First, after identifying equal cubics and discarding those that are reducible or have a fixed divisor, we assume that $\Theta(d^4)$ viable cubics remain in $\mathcal P_d$. The range $0\le d\le D$ then supplies $\Theta(D^5)$ candidate evaluations of these cubics, counted with multiplicity. Second, we assume that, apart from the collisions forced by their confinement to $O(mD^2)$ integers, these values are sufficiently well distributed that, by the time $D$ has order $m^{1/2}$, their distinct values contain all but a negligible proportion of the primes in $I_m$. Under these assumptions, almost every prime $p\in I_m$ satisfies $\Delta(p)\le m^{1/2+o(1)}$. Since $p$ has order $m^3$, this predicts $\Delta(p)\le p^{1/6+o(1)}$.

The assumptions on the number and distribution of viable cubic values are unproved, and not clearly available even under the Bateman--Horn conjecture. The argument therefore suggests, but does not prove, that $\Delta(p)=p^{1/6+o(1)}$ for almost all primes. Combining this with \Cref{cor:almost-all-deficits} gives the following conjecture.

\begin{conjecture}
    \label{conj:typical-deficit}
    For every $\varepsilon>0$,
    \[
        \lim_{X\to\infty}
        \frac{\#\{p\le X:p\text{ is prime and }
        p^{1/6-\varepsilon}\le\Delta(p)\le p^{1/6+\varepsilon}\}}{\pi(X)}
        =1.
    \]
\end{conjecture}

\newpage
\appendix
\section{Proof of \texorpdfstring{\Cref{prop:curve_tqf}}{Proposition~\ref{prop:curve_tqf}}}
\label{app:curve_tqf}
\begin{proof}
    Suppose that $E/\overline{\mathbb F}_p$ is a supersingular elliptic curve with $\delta_E=n$, and let $\mathcal O\subset B_{p,\infty}$ be a maximal order isomorphic to its endomorphism ring. Let $Q_E$ be the rational quadratic form on $\bigwedge^2\mathcal O^T$ obtained by scaling the determinant form by $1/(4p)$. In other words, for $\alpha,\beta\in\mathcal O^T$,
    \begin{equation}
        \label{eq:appendix-normalized-determinant}
        Q_E(\alpha\wedge\beta)
        =\frac{\det(\alpha,\beta)}{4p}.
    \end{equation}
    We will show that $Q_E$ is a positive definite integral ternary quadratic form whose values are exactly the normalized determinants appearing in \Cref{prop:aov_delta_E}, and then compute its discriminant.
    Write $\alpha=2x-\operatorname{trd}(x)$ and $\beta=2y-\operatorname{trd}(y)$ for some $x,y\in\mathcal O$. Since scalars commute with every element of $\mathcal O$,
    \[
        [\alpha,\beta]
        =[2x-\operatorname{trd}(x),2y-\operatorname{trd}(y)]
        =4[x,y],
    \]
    so $[\alpha,\beta]/4\in [\mathcal O,\mathcal O]$. The commutator ideal is the unique two-sided prime ideal of $\mathcal O$ above $p$, so the reduced norm of this element is divisible by $p$ \cite[Section~42.4.6]{voight_quaternion_2021}.
    We now compute the reduced norm of $[\alpha,\beta]/4$ in terms of the determinant form. Let $t\colonequals\operatorname{trd}(\alpha\overline{\beta})/2$. Since $\alpha$ and $\beta$ have reduced trace zero, we have $\operatorname{trd}(\alpha\beta)=-2t$ and $\alpha\beta+\beta\alpha=-2t$.
    Thus,
    \[
        \frac{[\alpha,\beta]}{4}
        =\frac{\alpha\beta+t}{2}.
    \]
    Using $\operatorname{nrd}(q+a)=\operatorname{nrd}(q)+a\operatorname{trd}(q)+a^2$ for $a\in\mathbb Q$, we obtain
    \begin{align*}
        4\operatorname{nrd}\left(\frac{[\alpha,\beta]}{4}\right)
        &=\operatorname{nrd}(\alpha\beta+t)\\
        &=\operatorname{nrd}(\alpha)\operatorname{nrd}(\beta)
        +t\operatorname{trd}(\alpha\beta)+t^2\\
        &=\operatorname{nrd}(\alpha)\operatorname{nrd}(\beta)-t^2
        =\det(\alpha,\beta).
    \end{align*}
    Therefore,
    \[
        Q_E(\alpha\wedge\beta)
        =\frac{1}{p}\operatorname{nrd}\left(\frac{[\alpha,\beta]}{4}\right)
        =\frac{\operatorname{nrd}([x,y])}{p}.
    \]
    So, for every $\alpha,\beta\in\mathcal O^T$, the normalized commutator pairing identifies $Q_E(\alpha\wedge\beta)$ with a reduced norm divided by $p$; in particular, this value is integral.

    Fix an ordered basis $e_1,e_2,e_3$ of $\mathcal O^T$, and let $G_{\mathcal O^T}$ be the Gram matrix of the reduced norm. With respect to the induced basis
    \[
        (e_2\wedge e_3,e_3\wedge e_1,e_1\wedge e_2)
    \]
    of $\bigwedge^2\mathcal O^T$, the determinant form has Gram matrix $\operatorname{adj}(G_{\mathcal O^T})$. Hence $Q_E$ has Gram matrix
    \[
        \frac{1}{4p}\operatorname{adj}(G_{\mathcal O^T}).
    \]
    Since the reduced norm is positive definite on $B_{p,\infty}$, both $G_{\mathcal O^T}$ and its adjugate are positive definite. Therefore, $Q_E$ is positive definite.

    In rank $3$, every element of $\bigwedge^2\mathcal O^T$ has the form $\alpha\wedge\beta$ for some $\alpha,\beta\in\mathcal O^T$. Hence $Q_E$ is integral. Equation~\eqref{eq:appendix-normalized-determinant} then identifies its nonzero values with the normalized determinants in \Cref{prop:aov_delta_E}. Therefore, $\min(Q_E)=\delta_E=n$.

    Finally, $\det(G_{\mathcal O^T})=4p^2$ \cite[Proposition~4.1]{aubry_minimal_2026}, and therefore
    \[
        \disc(Q_E)
        =4\det\left(\frac{\operatorname{adj}(G_{\mathcal O^T})}{4p}\right)
        =\frac{4\det(G_{\mathcal O^T})^2}{(4p)^3}
        =p,
    \]
    which proves the forward direction.

    Conversely, let $Q$ be a positive definite integral ternary quadratic form with $\min(Q)=n$ and $\disc(Q)=p$. Let $\mathcal O$ be the even Clifford algebra
    \[
        \operatorname{Clf}^0(Q)
        \colonequals
        \operatorname{span}_{\ZZ}
        \left\{
            \vec x_1\cdots\vec x_{2m}
            :m\ge 0,\ \vec x_i\in\ZZ^3
        \right\}
        \subseteq \operatorname{Clf}(Q),
    \]
    where
    \[
        \operatorname{Clf}(Q)
        \colonequals
        T(\ZZ^3)\Big/
        \left\langle
        \vec x\otimes\vec x-Q(\vec x):\vec x\in\ZZ^3
        \right\rangle,
    \]
    and $T(\ZZ^3)$ is the tensor algebra of $\ZZ^3$.

    The even Clifford correspondence makes $\mathcal O$ a quaternion order of reduced discriminant $\disc(Q)=p$ \cite[Main Theorem~22.1.1]{voight_quaternion_2021}. Set $B=\mathcal O\otimes_{\ZZ}\mathbb Q$. Because $Q$ is positive definite, $B$ is ramified at $\infty$. Moreover, the discriminant of $B$ divides the reduced discriminant $p$ of $\mathcal O$. It cannot equal $1$, since the quaternion algebra over $\mathbb Q$ of discriminant $1$ is split. Therefore, $B\cong B_{p,\infty}$, and $\mathcal O$ is maximal because its reduced discriminant equals the discriminant of $B$ \cite[Theorem~15.5.5]{voight_quaternion_2021}.

    We now show that the normalized determinant form associated with $\mathcal O$ is integrally equivalent to $Q$. Let
    \[
        Q=aX^2+bY^2+cZ^2+rYZ+sXZ+tXY,
    \]
    so that
    \[
        G_Q=
        \begin{pmatrix}
            a&t/2&s/2\\
            t/2&b&r/2\\
            s/2&r/2&c
        \end{pmatrix}.
    \]
    Let $1,i,j,k$ be the Clifford basis of $\mathcal O$ from \cite[(22.1.2)]{voight_quaternion_2021}. For an element $\alpha=q_0+xi+yj+zk$, the trace and reduced norm are
    \begin{align*}
        \operatorname{trd}(\alpha)
        ={}&2q_0+rx+sy+tz,\\
        \operatorname{nrd}(\alpha)
        ={}&q_0^2+rq_0x+sq_0y+tq_0z
        +bcx^2+acy^2+abz^2\\
        &+(rs-ct)xy+(rt-bs)xz+(st-ar)yz,
    \end{align*}
    by \cite[(22.3.18)]{voight_quaternion_2021}. In particular, $\operatorname{trd}(i)=r$, $\operatorname{trd}(j)=s$, and $\operatorname{trd}(k)=t$. Recall that $\tau(x)=2x-\operatorname{trd}(x)$ induces an isomorphism $\mathcal O/\ZZ\cong\mathcal O^T$. Since $i,j,k$ form a basis of $\mathcal O/\ZZ$, their images
    \[
        \tau(i)=2i-r,
        \qquad
        \tau(j)=2j-s,
        \qquad
        \tau(k)=2k-t
    \]
    form a basis of $\mathcal O^T$. To compute $G_{\mathcal O^T}$ with respect to the ordered basis $(\tau(i),\tau(j),\tau(k))$, let $u,v,w\in\ZZ$ and set
    \[
        \gamma=u\tau(i)+v\tau(j)+w\tau(k)
        =-(ru+sv+tw)+2ui+2vj+2wk.
    \]
    Substituting $q_0=-(ru+sv+tw)$, $x=2u$, $y=2v$, and $z=2w$ into the reduced-norm formula above yields
    \begin{align*}
        \operatorname{nrd}(\gamma)
        ={}&-(ru+sv+tw)^2
        +4bcu^2+4acv^2+4abw^2\\
        &+4(rs-ct)uv+4(rt-bs)uw+4(st-ar)vw\\
        ={}&(4bc-r^2)u^2+(4ac-s^2)v^2+(4ab-t^2)w^2\\
        &+2(rs-2ct)uv+2(rt-2bs)uw+2(st-2ar)vw.
    \end{align*}
    Reading off the coefficients, we obtain
    \[
        G_{\mathcal O^T} =
        \begin{pmatrix}
            4bc-r^2&rs-2ct&rt-2bs\\
            rs-2ct&4ac-s^2&st-2ar\\
            rt-2bs&st-2ar&4ab-t^2
        \end{pmatrix} = 4\operatorname{adj}(G_Q).
    \]
    Thus, with respect to the basis
    \[
        (\tau(j)\wedge\tau(k),
            \tau(k)\wedge\tau(i),
        \tau(i)\wedge\tau(j)),
    \]
    the normalized determinant form on $\bigwedge^2\mathcal O^T$ has Gram matrix
    \[
        \frac{1}{4p}\operatorname{adj}(G_{\mathcal O^T})
        =\frac{1}{4p}\operatorname{adj}\left(4\operatorname{adj}(G_Q)\right)
        =\frac{4\det(G_Q)}{p}G_Q
        =G_Q,
    \]
    where the last equality uses $\det(G_Q)=\disc(Q)/4=p/4$. Thus, the normalized determinant form attached to $\mathcal O$ is integrally equivalent to $Q$.

    By the Deuring correspondence, there is a supersingular elliptic curve $E/\overline{\mathbb F}_p$ whose endomorphism ring is isomorphic to $\mathcal O$. This isomorphism identifies the normalized determinant form associated with $E$ with that attached to $\mathcal O$, and hence with $Q$. Therefore, \Cref{prop:aov_delta_E} gives $\delta_E=\min(Q)=n$, as desired.
\end{proof}

\section{A 256-Bit Example for Remark~\ref{rem:param_fam_alg}}
\label{app:param_ex}
\begin{example}
    The computational supplement \cite{swanson_computational_2026} implements the construction in \Cref{rem:param_fam_alg}. Taking
    \begin{align*}
        n&=19721733612227209092649015,\\
        D&=1555787106686609923771791384727601443637276292476959,\\
        t&=37265851146608196127291335
    \end{align*}
    gives a $256$-bit prime $p=Dt-(n+1)$. The computation returns a supersingular curve $E/\mathbb F_{p^2}$ with $\delta_E=n\approx 2^{84}$.
    If $z$ denotes the field generator used in the supplement, then $j(E)=uz+v$, where
    \[
        \resizebox{\linewidth}{!}{$
            \begin{aligned}
                u&=46191504497042202305927789725709644307316017629766727501491860056507417597253,\\
                v&=54141910635733036816886759961912771019391613041790050170790834885492281030672.
            \end{aligned}
        $}
    \]
    Finding the parameters and maximal order took $0.388$ seconds; the constructive Deuring step, performed using the implementation accompanying \cite{cremona_deuring_2024}, took $13$ minutes.

\end{example}

\section{Details for \texorpdfstring{\Cref{lem:a3}}{Lemma~\ref{lem:a3}}}
\label{app:a3-claims}

\noindent\textbf{Claim~\ref{clm:a3-domain-comparison}.}
Suppose $0\le\epsilon\le3/50$. For every $(u,v,w)\in D_\epsilon$, there is some $u_D\in\RR$ such that $(u_D,v,w)\in D$ with $|u-u_D|\le\epsilon$ and
\[
    |R(u_D,v,w)-R(u,v,w)|\le4\epsilon.
\]

\begin{proof}
    Fix $(u,v,w)\in D_\epsilon$ and set $u_{\min}\colonequals\max\{-1,v+w-2\}$. Then $D$ contains precisely the points $(t,v,w)$ with $u_{\min}\le t\le1$. On the other hand, the inequalities defining $D_\epsilon$ give
    \[
        u_{\min}-\epsilon\le u\le1+\epsilon.
    \]
    Thus, there is some $u_D\in[u_{\min},1]$ such that $|u-u_D|\le\epsilon$.

    For every $t$ between $u$ and $u_D$, we have $|t|\le1+\epsilon$ and hence
    \[
        \left|\frac{\partial R}{\partial u}(t,v,w)\right|
        =|-2t+vw|
        \le3+2\epsilon<4.
    \]
    The mean value theorem now gives
    \[
        |R(u_D,v,w)-R(u,v,w)|\le4|u_D-u|\le4\epsilon,
    \]
    as desired.
\end{proof}

\noindent\textbf{Claim~\ref{clm:a3-distance}.}
If $(u,v,w)\in D$, then
\[
    d_1((u,v,w),\mathcal A)\le4R(u,v,w).
\]

\begin{proof}
    We divide $D$ according to the sign of $u$. Suppose first that $0\le u,v,w\le1$. In this region, we may discard the two points of $\mathcal A$ with first coordinate $-1$. Indeed, since $|u-1|\le|u+1|$, the point $(1,0,1)$ is at least as close as $(-1,0,1)$, and $(1,1,0)$ is at least as close as $(-1,1,0)$. Hence we may choose a nearest point among the four points of $\mathcal A$ in $\{0,1\}^3$. Looking at the points in $\mathcal A$, we end up with two options: either $(1,1,1)$ is a nearest point, or the nearest point has exactly two $1$s.

    In the first case, we write
    \[
        (u,v,w)=(1-\alpha,1-\beta,1-\gamma),
    \]
    where $\alpha,\beta,\gamma\ge0$. Comparing with $(0,1,1)$, $(1,0,1)$, and $(1,1,0)$ shows that $\alpha,\beta,\gamma\le1/2$. Expanding $R$ gives
    \begin{align*}
        R(u,v,w)
        ={}&\alpha+\beta+\gamma-\alpha^2-\beta^2-\gamma^2
        +\alpha\beta+\alpha\gamma+\beta\gamma-\alpha\beta\gamma\\
        \ge{}&\alpha+\beta+\gamma-\alpha^2-\beta^2-\gamma^2
        \ge\frac12(\alpha+\beta+\gamma).
    \end{align*}

    In the second case, after permuting the coordinates, we may assume that $(1,1,0)$ is a nearest point and write
    \[
        (u,v,w)=(1-\alpha,1-\beta,\gamma).
    \]
    Comparing its distance with the distances to $(1,1,1)$, $(0,1,1)$, and $(1,0,1)$ yields
    \[
        \gamma\le\frac12,
        \qquad
        \alpha+\gamma\le1,
        \qquad
        \beta+\gamma\le1.
    \]
    Therefore,
    \begin{align*}
        R(u,v,w)
        ={}&(2\alpha-\alpha^2-\alpha\gamma) +(2\beta-\beta^2-\beta\gamma)\\
        &+(\gamma-\gamma^2)+\alpha\beta\gamma
        \ge\alpha+\beta+\frac\gamma2
        \ge\frac12(\alpha+\beta+\gamma).
    \end{align*}
    In each case, $\alpha+\beta+\gamma$ is the $L^1$-distance to the chosen point of $\mathcal A$. Thus,
    \[
        d_1((u,v,w),\mathcal A)\le2R(u,v,w)\le4R(u,v,w),
    \]
    as desired.

    It remains to consider $-1\le u\le0$. Set $q=-u$. The condition $(u,v,w)\in D$ becomes
    \[
        0\le q,v,w\le1
        \qquad\text{and}\qquad
        q+v+w\le2.
    \]
    A nearest point of $\mathcal A$ is then one of $(0,1,1)$, $(-1,0,1)$, and $(-1,1,0)$. Among these three points, the zero coordinate of a nearest point occurs in a position where the corresponding coordinate of $(q,v,w)$ is smallest. Permuting $q,v,w$ preserves both $R(-q,v,w)$ and the displayed conditions, while permuting the distances to these three points. We may therefore suppose that $q\le v,w$. The corresponding nearest point is $(0,1,1)$, and its distance is $2-v-w+q$.

    On this region, \eqref{eq:a3-distance-claim} is equivalent to the nonnegativity of the difference
    \[
        H(v,w,q)\colonequals4R(-q,v,w)-2+v+w-q
    \]
    on the polytope
    \[
        P\colonequals
        \{(v,w,q)\in[0,1]^3:v+w+q\le2,\ q\le v,\ q\le w\}.
    \]
    For $(v,w,q)\in P$, we have
    \[
        -\frac14\nabla^2H=
        \begin{pmatrix}
            2&q&w\\
            q&2&v\\
            w&v&2
        \end{pmatrix},
    \]
    which is positive semidefinite since it is symmetric and diagonally dominant with nonnegative diagonal entries. Thus, $H$ is concave on $P$.
    The seven vertices of $P$ are thus
    \[
        (0,0,0),(1,0,0),(0,1,0),(1,1,0),
        \left(1,\tfrac12,\tfrac12\right),
        \left(\tfrac12,1,\tfrac12\right),
        \left(\tfrac23,\tfrac23,\tfrac23\right),
    \]
    where $H$ takes the values $6,3,3,0,0,0$, and $4/27$, respectively. Since these values are nonnegative, concavity thus gives $H\ge0$ on $P$, and this proves the claim.
\end{proof}


\noindent\textbf{Claim~\ref{clm:a3-chart-change}.}
The form $Q$ is integrally equivalent to $nQ_0+h$, where
\[
    h\colonequals aY^2+bZ^2-rYZ-sXZ-tXY
\]
for some integers satisfying $0\le a\le b<57e$ and $|r|,|s|,|t|<57e$.

\begin{proof}
    Recall that we start with some $\vec c=(c_1,c_2,c_3)\in\mathcal A$ and integers $k_1,k_2,k_3$ such that the cross coefficients of $Q$ are $nc_i+k_i$, with
    \[
        0\le a_0\le b_0,
        \quad
        a_0+b_0\le3e,
        \quad
        |k_1|+|k_2|+|k_3|<51e.
    \]
    We now wish to move to the $Q_0$-chart by a unimodular change of coordinates that fixes $(1,0,0)$. For each $\vec c\in\mathcal A$, define
    \[
        Q_{\vec c}\colonequals X^2+Y^2+Z^2+c_1YZ+c_2XZ+c_3XY
    \]
    so that $Q=nQ_{\vec c}+h_0$ with
    \[
        h_0\colonequals a_0Y^2+b_0Z^2+k_1YZ+k_2XZ+k_3XY.
    \]
    For each $\vec c\in\mathcal A$, \Cref{tab:a3-changes} lists a change of variables $\phi_{\vec c}$ that carries $Q_{\vec c}$ to $Q_0$. Its effect on $h_0$ is recorded by writing
    \[
        h_0\circ\phi_{\vec c}
        =AY^2+BZ^2+K_1YZ+K_2XZ+K_3XY.
    \]
    \begin{table}[t]
        \centering
        \renewcommand{\arraystretch}{1.25}
        \[
            \begin{array}{cccc}
                \toprule
                \vec c&\phi_{\vec c}(X,Y,Z)&(A,B)&(K_1,K_2,K_3)\\
                \midrule
                (1,1,1)
                &(X,Y,Z)
                &(a_0,b_0)
                &(k_1,k_2,k_3)\\
                (0,1,1)
                &(X+Z,Y,-Z)
                &(a_0,b_0-k_2)
                &(k_3-k_1,-k_2,k_3)\\
                (1,0,1)
                &(X,Y+Z,-Z)
                &(a_0,a_0+b_0-k_1)
                &(2a_0-k_1,k_3-k_2,k_3)\\
                (1,1,0)
                &(X,-Y,Y+Z)
                &(a_0+b_0-k_1,b_0)
                &(2b_0-k_1,k_2,-k_3+k_2)\\
                (-1,0,1)
                &(X,Y+Z,Z)
                &(a_0,a_0+b_0+k_1)
                &(2a_0+k_1,k_3+k_2,k_3)\\
                (-1,1,0)
                &(X,Y,Y+Z)
                &(a_0+b_0+k_1,b_0)
                &(2b_0+k_1,k_2,k_3+k_2)\\
                \bottomrule
            \end{array}
        \]
        \caption{Changes of variables carrying each $Q_{\vec c}$ to $Q_0$.}
        \label{tab:a3-changes}
    \end{table}

    Each change of variables is unimodular, fixes $(1,0,0)$ and by construction satisfies
    \[
        Q\circ\phi_{\vec c}
        =nQ_0+AY^2+BZ^2+K_1YZ+K_2XZ+K_3XY.
    \]
    In particular, the formulas in the table, together with $a_0\le b_0$ and the starting bounds, give
    \[
        |A|,|B|,|K_1|,|K_2|,|K_3|
        \le |k_1|+|k_2|+|k_3|+2b_0<57e.
    \]

    It still remains to show $0\le A\le B$. Since $Q\circ\phi_{\vec c}$ has minimum $n$, we have
    \[
        n+A=(Q\circ\phi_{\vec c})(0,1,0)\ge n,
        \qquad
        n+B=(Q\circ\phi_{\vec c})(0,0,1)\ge n.
    \]
    Thus, $A,B\ge0$, and interchanging $Y$ and $Z$ if necessary gives $A\le B$. Renaming $A,B,K_1,K_2,K_3$ to $a,b,-r,-s,-t$, respectively, proves the claim.
\end{proof}

\section{Extending the Cubic Parameterizations Below the Uniform Bound}
\label{app:prime-discriminant-bounds}

\subsection{The General Fixed-Deficit Bound}
\label{app:fixed-deficit-bound}

Fix $d\ge0$. We describe a finite procedure for determining the least integer $N_d$ such that, for every $n\ge N_d$ and every prime $p$ satisfying
\[
    2(n+d)^3<p<2(n+d+1)^3,
\]
there is a positive definite integral ternary quadratic form of minimum $n$ and discriminant $p$ if and only if $p=P_\lambda(n)$ for some $\lambda\in\Lambda_d$.

We divide the computation into two ranges. Set $e=d+1$. Suppose first that $n\ge25e^2$, and let $Q$ be a positive definite integral ternary quadratic form satisfying
\[
    \min(Q)=n,
    \qquad
    2(n+d)^3<\disc(Q)<2(n+e)^3.
\]
By \Cref{lem:a3}, we may choose an equivalent representative $Q=nQ_0+h_\lambda$ with $\lambda\in\ZZ_{\ge0}^6$ and $\lambda_{\min}=0$. After applying the $S_4$-action, we may assume that the first coordinate of $\lambda$ is zero and write
\[
    Q=nQ_0+aY^2+bZ^2-rYZ-sXZ-tXY,
\]
so that
\[
    \lambda=(0,a,b,a+t,b+s,a+b+r)
    \quad\text{and}\quad
    \wt(\lambda)=3a+3b+r+s+t.
\]
A direct expansion gives
\begin{align*}
    \disc(Q)&=2n^3+\wt(\lambda)n^2+C_1(\lambda)n+C_0(\lambda),\\
    C_1(\lambda)&=rs+rt+st-r^2-s^2-t^2+2bt+2as+4ab,\\
    C_0(\lambda)&=-rst-bt^2-as^2.
\end{align*}
Thus, the two bounds on the discriminant are equivalent to
\begin{align}
    (\wt(\lambda)-6d)n^2
    +(C_1(\lambda)-6d^2)n+C_0(\lambda)-2d^3&>0,\\
    \label{eq:fixed-deficit-upper-bound}
    (6e-\wt(\lambda))n^2
    +(6e^2-C_1(\lambda))n+2e^3-C_0(\lambda)&>0.
\end{align}
The first inequality, together with \Cref{lem:weight-trace}, gives $\wt(\lambda)>6d$. If $\wt(\lambda)\le6e$, then $\lambda\in\Lambda_d$, so $\disc(Q)=P_\lambda(n)$ is already detected by $\mathcal P_d$. Thus, the only possible exceptions have $\wt(\lambda)>6e$. For these tuples, the coefficient of $n^2$ in \eqref{eq:fixed-deficit-upper-bound} is negative, but the lower-order terms may still place the discriminant in the target interval when $n$ is small. It remains to find these exceptions.

The proof of \Cref{lem:a3} gives a finite search space for the possible perturbations. In the coordinates used there, they arise from integers satisfying
\[
    0\le a_0\le b_0,\quad
    a_0+b_0\le3e,\quad\text{and}\quad
    |k_1|+|k_2|+|k_3|\le51e-1.
\]
There are $O(e^2)$ choices for $(a_0,b_0)$ and
\[
    1+6\binom{51e-1}{1}+12\binom{51e-1}{2}
    +8\binom{51e-1}{3}=O(e^3)
\]
choices for $(k_1,k_2,k_3)$. For each choice and each of the six $A_3$ centers, we apply the corresponding change of variables in \Cref{tab:a3-changes} and interchange $Y$ and $Z$ when necessary. We retain the resulting tuple only when it belongs to $\ZZ_{\ge0}^6$ and has minimum zero. After removing duplicates, this gives a list of $O(e^5)$ perturbations containing every possibility.

For each tuple in this list with $\wt(\lambda)>6e$, we solve the two discriminant inequalities exactly and enumerate the integers in the intersection of their solution sets. For each such integer $n$, we test whether
\[
    \min(nQ_0+h_\lambda)=n
    \qquad\text{and}\qquad
    \disc(nQ_0+h_\lambda)\text{ is prime}.
\]
This finds all exceptions with $n\ge25e^2$.

For $n<25e^2$, we instead enumerate the sign-normalized Eisenstein-reduced forms in the target discriminant interval and compare their prime discriminants with the values $P_\lambda(n)$ for $\lambda\in\Lambda_d$. Combining the two ranges finds every exception. Therefore, $N_d$ is one more than the largest value of $n$ that occurs.

For a fixed $D\ge0$, this gives a single bound for the algorithm in \Cref{cor:deficit-alg}. Set
\[
    M_D\colonequals\max_{0\le d\le D}(N_d+d).
\]
If $p\ge2M_D^3$ and $m\colonequals\lfloor\sqrt[3]{p/2}\rfloor$, then $m-d\ge N_d$ for every $0\le d\le D$. Hence, we may replace the uniform lower bound in \Cref{cor:deficit-alg} by $p\ge2M_D^3$. Omitting the primality test gives the analogous bound for all discriminants.

\subsection{The \texorpdfstring{$d=0$}{d=0} Computation}
\label{app:d0-computation}

We now apply this procedure with $d=0$ in the range $n\ge25$. The computation finds $42$ distinct pairs $(p,n)$ arising from perturbations of weight at least $7$ for which the form has minimum $n$ and prime discriminant $p$. All but one satisfy $p=P_i(n)$ for some $1\le i\le9$; the exception is
\[
    (p,n)=(39313,26).
\]
The computational supplement \cite{swanson_computational_2026} implements this enumeration.

Together with the analysis of $\mathcal P_0$ in the proof of \Cref{thm:wisde-classification}, this shows that every wisde prime in the remaining range $27\le n<3000$ has the form $P_i(n)$ for some $1\le i\le9$.

It remains to prove the converse in this range. Suppose that $p=P_i(n)$ for some $1\le i\le9$, and consider the following forms.
\begin{equation}
    \label{eq:candidate-forms}
    \begin{array}{clc}
        \toprule
        i&Q_i&\lambda_i\\
        \midrule
        1&nQ_0+Z^2+XZ-2XY&(0,0,1,2,0,1)\\
        2&nQ_0+Z^2-YZ-XY&(0,0,1,1,1,2)\\
        3&nQ_0+Y^2+Z^2+2YZ-2XZ+XY&(0,1,1,0,3,0)\\
        4&nQ_0+Z^2-YZ-2XY&(0,0,1,2,1,2)\\
        5&nQ_0+2Z^2+YZ+2XZ-3XY&(0,0,2,3,0,1)\\
        6&nQ_0+Y^2+Z^2-YZ+XY&(0,1,1,0,1,3)\\
        7&nQ_0+Y^2+Z^2-XZ+XY&(0,1,1,0,2,2)\\
        8&nQ_0+Y^2+Z^2+2YZ-3XZ+XY&(0,1,1,0,4,0)\\
        9&nQ_0+Y^2+2Z^2+2XZ+XY&(0,1,2,0,0,3)\\
        \bottomrule
    \end{array}
\end{equation}
Each vector $\lambda_i$ belongs to $\ZZ_{\ge0}^6$, satisfies $(\lambda_i)_{\min}=0$, and has weight at most $6$. Since $n\ge27$, \Cref{lem:wtmin} gives $\min(Q_i)=n$, while a direct calculation gives $\disc(Q_i)=P_i(n)$. The formulas in \eqref{eq:cubic-families} also show that
\[
    2n^3<P_i(n)<2(n+1)^3
\]
for every $n\ge27$. Thus, if $p=P_i(n)$ is prime, \Cref{prop:curve_tqf} gives a supersingular elliptic curve $E/\overline{\mathbb F}_p$ with $\delta_E=n$. The displayed inequalities give $\lfloor\sqrt[3]{p/2}\rfloor=n$, so the general upper bound for $\delta(p)$ is attained and $p$ is wisde.

Finally, the exceptional pair $(39313,26)$ shows that the hypothesis $n\ge27$ in \Cref{thm:wisde-classification} is sharp. Thus, in the notation of \Cref{app:fixed-deficit-bound}, $N_0=27$.

\newpage
\bibliographystyle{amsplain}
\bibliography{refs}

\end{document}